\documentclass[11pt]{article}

\usepackage[english]{babel}
\usepackage{amsmath,amsfonts,amssymb,amsthm}
\usepackage{latexsym}
\usepackage{makeidx}

\usepackage{vmargin}
\setmarginsrb{11.6mm}{10mm}{11.6mm}{20mm}
            {8mm}{8mm}{8mm}{8mm}

\numberwithin{equation}{section}

\numberwithin{table}{section}

\usepackage{caption}

\newtheorem{deff}{Definition}[section]
\newtheorem{lemma}[deff]{Lemma}
\newtheorem{theorem}[deff]{Theorem}
\newtheorem{corollary}[deff]{Corollary}

\newtheorem{proposition}[deff]{Proposition}
\newtheorem{fact}[deff]{Fact}
\newtheorem{em-example}[deff]{Example}
\newtheorem{em-def}[deff]{Definition}        
\newtheorem{em-remark}[deff]{Remark}         
\newtheorem{em-question}[deff]{Question}
\newtheorem{question}[deff]{Question}

\newtheorem{problem}[deff]{Problem}

\newenvironment{example}{\begin{em-example} \em }{ \end{em-example}}
\newenvironment{remark}{\begin{em-remark} \em }{\end{em-remark}}

\newcommand{\R}{\mathbb R}
\newcommand{\N}{\mathbb N}

\newcommand{\Q}{\mathbb Q}
\newcommand{\J}{\mathbb J}
\newcommand{\Z}{\mathbb Z}
\newcommand{\T}{\mathbb{T}}

\newcommand{\Gs}{\mathfrak G}

\newcommand{\CC}{\mathcal C}
\newcommand{\cont}{\mathfrak c}

\def\ent{\mathrm{ent}}
\def\aent{\mathrm{ent}^\star}

\def\Hom{\mathrm{Hom}}
\def\Ab{\mathbf{Ab}}
\def\TopAb{\mathbf{TopAb}}

\input xy
\xyoption{all}

\title{Functorial topologies and finite-index subgroups of abelian groups}

\author{ Dikran Dikranjan
\\{\footnotesize {\tt  dikran.dikranjan@uniud.it}} 
 \\{\footnotesize Dipartimento di Matematica e Informatica,}
\\{\footnotesize Universit\`{a} di Udine,}
\\{\footnotesize Via delle Scienze, 206 - 33100 Udine, Italy}
 \and Anna Giordano Bruno
\\{\footnotesize {\tt  anna.giordanobruno@uniud.it}} 
 \\{\footnotesize Dipartimento di Matematica e Informatica,}
\\{\footnotesize Universit\`{a} di Udine,}
\\{\footnotesize Via delle Scienze, 206 - 33100 Udine, Italy}
 }

\date{Dedicated to the seventieth birthday of Eraldo Giuli}

\begin{document}

\maketitle

\abstract{In the general context of functorial topologies, we prove that in the lattice of all group topologies on an abelian group, the infimum between the Bohr topology and the natural topology is the profinite topology. The profinite topology and its connection to other  functorial topologies is the main objective of the paper. We are particularly interested
in the poset $\mathcal C(G)$ of all finite-index subgroups of an abelian group $G$, since it is a local base for the profinite topology of $G$. 
We describe various features of the poset $\mathcal C(G)$ (its cardinality, its cofinality, etc.) and we characterize the abelian groups $G$ for which $\mathcal C(G)\setminus \{G\}$ is cofinal in the poset of all subgroups of $G$ ordered by inclusion. Finally, for pairs of functorial topologies $\mathcal T,\mathcal S$ we define the equalizer $\mathcal E(\mathcal T,\mathcal S)$, which permits to describe relevant classes of abelian groups in terms of functorial topologies.
}

\section{Introduction}

The concept of functorial topology was introduced by Charles \cite{Cha2}. He proposed a method for constructing such topologies, which was generalized by Fuchs \cite[Vol. I, p. 33]{F}. Later on, functorial topologies were subject of study by many authors, among them Boyer and Mader \cite{BM}, Mader \cite{Ma1} and Mader and Mines \cite{MM1}.

\smallskip
Let $U$ be the forgetful functor $U:\TopAb \to \Ab$, where $\Ab$ is the category of all abelian groups and their morphisms and $\TopAb$ is the category of all topological abelian groups and their morphisms.

\begin{deff}
A functorial topology is a functor $\mathcal T:\Ab\to\TopAb$ such that $U\mathcal T=1_\Ab$.
\end{deff}

Equivalently, following \cite{F}, a functorial topology is a class $\mathcal T=\{\mathcal T_G:G\in\Ab\}$, where $(G,\mathcal T_G)$ is a topological group for every $G\in\Ab$, and every homomorphism $G\to H$ in $\Ab$ is continuous $(G,\mathcal T_G)\to(H,\mathcal T_H)$.  So, a functorial topology is a functor $\mathcal T:\Ab\to\TopAb$ such that $\mathcal T(G)=(G,\mathcal T_G)$ for any $G\in\Ab$, where $\mathcal T_G$ denotes the topology on $G$, and $\mathcal T(\phi)=\phi$ for any morphism $\phi$ in $\Ab$ \cite{BM}. 

By \cite[Theorem 2.2]{BM}, any functorial topology defined on a full subcategory of $\Ab$ extends to $\Ab$, so that there is no need to introduce functorial topologies for full subcategories of $\Ab$. 

A functorial topology $\mathcal T$ is \emph{linear} if $\mathcal T_G$ is linear for every $G\in\Ab$ (recall that a group topology is \emph{linear} if it has a local base consisting of open subgroups); moreover, $\mathcal T$ is \emph{ideal}  if $\mathcal T$ maps surjective homomorphisms to open (continuous) homomorphisms \cite{BM}, and it is \emph{hereditary} if $\mathcal T_H=\mathcal T_G\restriction_H$ for every $G\in\Ab$ and every subgroup $H$ of $G$.

\smallskip
Functorial topologies exist in abundance. The discrete topology $\delta$ and the indiscrete topology $\iota$ are functorial topologies that are both hereditary and ideal.  In this paper we mainly focus our attention on the following three functorial topologies: for an abelian group $G$,
\begin{itemize}
\item the \emph{profinite topology} $\gamma_G$ has all finite-index subgroups as a base of the neighborhoods of $0$;
\item the \emph{natural topology} (sometimes called also \emph{$\Z$-adic topology}) $\nu_G$ has the countable family of subgroups $\{mG:m\in \N_+\}$ as a base of the neighborhoods of $0$;
\item the \emph{Bohr topology} $\mathcal P_G$ is the initial topology of all homomorphisms $G \to \T=\R/\Z$, namely, the characters of $G$, where $\T$ is equipped with the compact quotient topology of $\R/\Z$.
\end{itemize}

We also consider the $p$-adic topology $\nu^p$, which can be viewed as a natural local version of $\nu$ and analogous local versions $\gamma^p$ of $\gamma$ and $\mathcal P^p$ of $\mathcal P$ (see Example \ref{examples}).

Unlike the Bohr topology, the profinite and the natural topology are linear topologies. On the other hand, the natural topology and the profinite topology are ideal but not hereditary (see Example \ref{QZ}), while  the Bohr topology is both ideal and hereditary (see Lemma \ref{PG}(b)). 

A topological abelian group $(G,\tau)$ is \emph{totally bounded} if for any non-empty open subset $U$ of $(G,\tau)$ there exists a finite subset $F$ of $G$ such that $U+F=G$. If $\tau$ is totally bounded and Hausdorff, it is said to be \emph{precompact}.
The completion $\widetilde G$ of any precompact abelian group $G$ is compact \cite{Weil}, so the precompact abelian groups are precisely the subgroups of the compact abelian groups. 
In this paper, we use the fact that the Bohr topology on an abelian group $G$ is the maximal totally bounded group topology on $G$; this is a deep fact, deducible from Peter-Weyl's theorem for compact abelian groups. See \cite{dLD}, \cite{DW}, \cite{kunen} and \cite{KR} for the remarkable properties of this topology. 

Our choice to concentrate mainly on the profinite, the natural and the Bohr topology, with a special emphasis on the connections between the Bohr topology and the profinite topology, is motivated by the fact that  the functorial aspect of the Bohr topology has not been sufficiently brought to light neither in topology nor in algebra.  

\medskip
In Section \ref{profinite} we investigate the properties of the profinite topology, mainly its relationship with the natural and the Bohr topology. 
In particular, it is known that $\gamma_G\leq \inf\{\nu_G,\mathcal P_G\}$ for any abelian group $G$ (for a proof see Lemma \ref{P,nu>gamma}). The first of the main theorems of this paper, which is Theorem \ref{inf=gamma}, shows that actually equality holds, that is, in the lattice of all group topologies of an abelian group, the profinite topology is the infimum of the natural topology and the Bohr topology. Since the natural topology is metrizable, one may be left with the misperception that the profinite topology and the Bohr topology are very close due to the equality $\gamma_G= \inf\{\nu_G,\mathcal P_G\}$. We see in Theorem \ref{str_div} that the behavior of the subgroups shows a substantial difference between these topologies. Namely, while every subgroup is closed in the Bohr topology, the abelian groups $G$ in which every subgroup is $\gamma_G$-closed  form a quite small class (this class consists precisely of the Pontryagin duals of the so called ``exotic tori" introduced in \cite{DP}). 

\bigskip
In \cite{DGS}, the adjoint algebraic entropy for endomorphisms $\phi$ of abelian groups $G$ was introduced making use of the family $\CC(G)$ of all finite-index subgroups of $G$ (see Section \ref{C(G)} for the precise definition). Indeed, the adjoint algebraic entropy of endomorphisms $\phi$ of abelian groups $G$ measures to what extent $\phi$ moves the finite-index subgroups of $G$. So, in the context of the adjoint algebraic entropy, it is worth studying the poset of finite-index subgroups of abelian groups $G$, calculating various invariants of it (as size, cofinality, etc.).  In this direction, \cite[Theorem 3.3]{DGS} (see Theorem \ref{narrow:char} below) characterizes the abelian groups $G$ with countable $\CC(G)$; abelian groups with this property are called \emph{narrow}. In particular, every endomorphisms $\phi$ of a narrow abelian group $G$ has $\ent^\star(\phi)=0$ \cite[Proposition 3.7]{DGS}. The surprising dichotomy discovered in \cite{DGS} (namely, $\CC(G)$ is either countable or has size at least $\cont$) is fully explained by Theorem \ref{WEIGHT} (see also Lemma \ref{profin*}), since $|\CC(G)|$ coincides with the size of the torsion part of a compact abelian group (namely, the Pontryagin dual of $G$). 

Indeed, in Section \ref{C(G)} we study the cardinality of $\CC(G)$ of an abelian group $G$ in the general setting.
The family $\CC(G)$ forms a semilattice with respect to intersections and with top element $G$. One can look at $\CC(G)$ also as a filter-base that gives rise to the profinite topology $\gamma_G$ of the abelian group $G$. More precisely, we show that when $\CC(G)$ is infinite, its cardinality coincides with the weight and the local weight of $(G,\gamma_G)$. So, the purely algebraic object $\CC(G)$ is strictly related to the topological invariants of the profinite topology of $G$ (see below for the definitions of these topological invariants).

Making use of the results from Section \ref{profinite}, for any abelian group $G$ we characterize the size and the cofinality of $\CC(G)$, that is, we compute the weight and the local weight of $(G,\gamma_G)$ in Theorem \ref{WEIGHT}.
Furthermore, Theorem \ref{density} characterizes the density character of $(G,\gamma_G)$, using the fact that it coincides with the density character and the weight of $(G,\nu_G)$.
In another direction, in Theorem \ref{str_div} we describe the abelian groups $G$ for which the poset $\mathcal C(G)\setminus \{G\}$ is cofinal in the larger poset $\mathcal S(G)$ of all subgroups of $G$ ordered by inclusion. 

\medskip
Inspired by the fact that the narrow abelian groups form precisely the class of abelian groups for which the profinite and the natural topology coincide (see Theorem \ref{narrow:char}), in Section \ref{constr-sec} we define the equalizer $\mathcal E(\mathcal T,\mathcal S)$ of a pair of functorial topologies $\mathcal T,\mathcal S$. 
Moreover, we describe its basic properties and arrive in this way to the standard correspondence between functorial topologies and classes of abelian groups stable under isomorphisms, finite products and subgroups. 
In this section we provide also more examples of functorial topologies to better illustrate the usefulness of the equalizer. 

\bigskip
We dedicate this paper to the seventieth birthday of Eraldo Giuli, for his relevant contributions in the field of categorical topology and in particular, the closure operators in the sense of  \cite{CDT,DG,DG1,DT}, of which the functorial topologies in the category of abelian groups are a relevant inspiring example.

\subsection*{Notation and terminology} 

We denote by $\mathbb Z$, $\mathbb N$, $\mathbb N_+$, $\Q$ and $\R$ respectively the set of integers, the set of natural numbers, the set of positive integers, the set of rationals and the set of reals. For $m\in\mathbb N_+$, we use $\mathbb Z(m)$ for the finite cyclic group of order $m$. 
Consider on $\T=\R/\Z$ the norm given by $||r+\Z||=\min\{d(r,m):m\in\Z\}$, with $d$ the usual metric of $\R$.

All groups in this paper are abelian. 
For an abelian group $G$ and $m\in \Z$, we let
$mG = \{mx: x\in G\}$ and $G^1=\bigcap_{m\in\N_+}m G$, the \emph{first Ulm subgroup} of $G$. 
Obviously, $(G/G^1)^1=0$. We say that an abelian group $G$ is {\em divisible}, if  $G^1=G$
(i.e., $G=mG$ for every $ m\in \N_+$). We denote by $D(G)$ the biggest divisible subgroup of $G$ (namely, the one generated by all divisible subgroups of $G$). 
Obviously, $D(G) \subseteq G^1$. We call $G$ {\em reduced} if $D(G)=0$. 
We denote by $r_0(G)$ the torsion-free rank of $G$ and, for a prime $p$, $r_p(G)$ denotes the $p$-rank of $G$, that is, $\dim_{\mathbb F_p}G[p]$, where $G[p]=\{x\in G:px =0\}$ is the $p$-socle of $G$ and $\mathbb F_p$ is the field with $p$ elements. More generally, for $n\in\N_+$, let $G[n]=\{x\in G:nx=0\}$.

For a subset $M$ of a topological space $X$, we denote by $\overline{M}$ the closure of $M$. 
For a topological group $(G,\tau)$ and a subgroup $H$ of $G$, let $\tau_q$ denote the quotient topology of $\tau$ on $G/H$. 
Moreover, the quotient group $(G/\overline{\{0\}},\tau_q)$ is the largest Hausdorff
quotient group of $(G,\tau)$; we call it the \emph{Hausdorff reflection} of $(G,\tau)$ (as $(G,\tau) \mapsto (G/\overline{\{0\}},\tau_q)$
defines a reflection of the category of all topological groups into the full subcategory of all Hausdorff topological groups). 

We denote by $\delta_G$ and $\iota_G$ respectively the discrete and the indiscrete topology of an abelian group $G$. 
We denote by $(G,\tau)^*$ the \emph{dual group} of a topological abelian group $(G,\tau)$, that is, $(G,\tau)^*$ is the abelian group of all continuous characters $(G,\tau)\to\mathbb T$, endowed with the discrete topology. In particular, $G^*=\Hom(G,\T)$.

For a topological abelian group $(G,\tau)$, the \emph{weight} $w(G,\tau)$ is the minimum cardinality of a base of $(G,\tau)$, and the \emph{local weight} (or, \emph{character}) $\chi(G,\tau)$ is the minimum cardinality of a local base of $(G,\tau)$. If $(G,\tau)$ is a totally bounded abelian group, then $w(G,\tau)=\chi(G,\tau)$ \cite{C}. 
Finally, the \emph{density character} $d(G,\tau)$ is the minimum cardinality of a dense subset of $(G,\tau)$.

\section{The profinite, the natural and the Bohr topology}\label{profinite}

The following are basic properties of functorial topologies.

\begin{lemma}\label{functorial-basic}
Let $\mathcal T$ be a functorial topology. Then:
\begin{itemize}
\item[(a)] $\mathcal T_{G_1\times G_2}=\mathcal T_{G_1}\times\mathcal T_{G_2}$ for every $G_1,G_2\in\Ab$;
\item[(b)] $\mathcal T_G\geq\prod_{i\in I}\mathcal T_{G_i}$ for every arbitrary family $\{G_i:i\in I\}$ in $\Ab$ with $G=\prod_{i\in I}G_i$;
\item[(c)] $\mathcal T_H\geq\mathcal T_G\restriction_H$ for every $G\in\Ab$ and every subgroup $H$ of $G$; equality holds for every $G$ and $H$ if and only if $\mathcal T$ is hereditary;
\item[(d)] $(\mathcal T_G)_q\geq\mathcal T_{G/H}$ for every $G\in\Ab$ and every subgroup $H$ of $G$; equality holds for every $G$ and $H$ if and only if $\mathcal T$ is ideal.
\end{itemize}
\end{lemma}
\begin{proof}
(a) Consider the projections $p_j:(G,\mathcal T_{G_1\times G_2})\to (G_j,\mathcal T_j)$ for $j=1,2$, which are continuous by the definition of functorial topology. Then, for every neighborhood $U_1\times U_2$ of $0$ in $(G_1\times G_2,\mathcal T_1\times \mathcal T_2)$ there exists a neighborhood $W$ of $0$ in $(G_1\times G_2,\mathcal T_{G_1\times G_2})$ such that $p_j(W)\subseteq U_j$ for $j=1,2$, that is, $W\subseteq U_1\times U_2$. Hence, $\mathcal T_{G_1\times G_2}\geq\mathcal T_{G_1}\times\mathcal T_{G_2}$. To prove the converse inequality consider the inclusions $i_j:(G_i,\mathcal T_{G_i})\to (G_1\times G_2,\mathcal T_{G_1\times G_2})$, for $j=1,2$, which are continuous by the definition of functorial topology. Then $\mathcal T_{G_1\times G_2}\leq \inf\{\mathcal T_{G_1}\times \delta_{G_2},\delta_{G_1}\times \mathcal T_{G_2}\}=\mathcal T_{G_1}\times\mathcal T_{G_2}$.

\smallskip
To prove (b) proceed as in the first part of the proof of item (a).
For (c) and (d) it suffices to note that by definition the inclusion $(H,\mathcal T_H)\hookrightarrow(G,\mathcal T_G)$ and  the projection $(G,\mathcal T_G)\to (G/H,\mathcal T_{G/H})$ are continuous.
\end{proof}

We introduce a partial order between functorial topologies by letting $\mathcal T \leq \mathcal S$ whenever $\mathcal T_G \leq \mathcal S_G$ for every abelian group $G$. This makes the class $\mathfrak F t$ of all functorial topologies a large complete lattice with top element $\delta$ and bottom element $\iota$.

\subsection{The profinite topology vs the natural topology}\label{profinite-natural}

There is an important connection (see \cite{F}) between the first Ulm subgroup $G^1$ of an abelian group $G$ and the family $\CC(G)$, namely
\begin{equation}\label{G^1}
G^1=\bigcap_{N\in\CC(G)}N.
\end{equation}

Now we recall a notion closely related to these two concepts:

\begin{deff}
An abelian group $G$ is \emph{residually finite} if $G$ is isomorphic to a subgroup of a direct product of finite abelian groups.
\end{deff}

Obviously, every residually finite abelian group is reduced.
\medskip

Clearly, the natural topology is metrizable whenever it is Hausdorff. As a consequence of \eqref{G^1}, Lemma \ref{Haus} shows that the profinite and the natural topology are simultaneously Hausdorff (respectively, indiscrete). Moreover, we see that this occurs precisely when the abelian group is residually finite (respectively, divisible) (which are merely algebraic properties).

\begin{lemma}\label{res-fin<->G1=0}\label{Haus}
Let $G$ be an abelian group. Then $G^1=\overline {\{0\}}^{\gamma_G}=\overline{\{0\}}^{\nu_G}$, 
so $(G/G^1,\gamma_{G/G^1})$ is the Hausdorff reflection of $(G,\gamma_G)$ and $(G/G^1,\nu_{G/G^1})$ is the Hausdorff reflection of $(G,\nu_G)$
Moreover, the following conditions are equivalent: 
\begin{itemize}
\item[(a)] $\gamma_G$ is Hausdorff (respectively, indiscrete); 
\item[(b)] $G$ is residually finite (respectively, divisible); 
\item[(c)] $G^1=0$ (respectively, $G^1= G$); 
\item[(d)] $\nu_G$ is Hausdorff (respectively, indiscrete). 
\end{itemize}
\end{lemma}
\begin{proof} 
The first two assertions are obvious and imply the equivalences (a)$\Leftrightarrow$(c)$\Leftrightarrow$(d). 

To prove that (b) and (c) are equivalent, first note that clearly $G^1=G$ precisely when $G$ is divisible.
Now assume that $G$ is residually finite. Then $G$ is isomorphic to a subgroup of $F=\prod_{i\in I}F_i$, where each $F_i$ is a finite abelian group. Since $F^1=0$, also $G^1=0$.
To prove the converse implication suppose that $G^1=0$. It follows from \eqref{G^1} that $G$ is isomorphic to a subgroup of $\prod_{N\in\CC(G)}G/N$, where each $G/N$ is obviously finite 
(for every $N\in\CC(G)$ consider the canonical projection $G\to G/N$; then the diagonal homomorphism $G\to \prod_{N\in\CC(G)}G/N$ is injective by the assumption $G^1=0$).
\end{proof}

\begin{remark}\label{G1} 
Let $G$ be an abelian group. Every $N\in\CC(G)$ contains $G^1$, and the canonical projection
$\pi:G \to G/G^1$ gives rise to a bijection between $\CC(G/G^1)$ and $\CC(G)$ by taking inverse images under $\pi$.
\end{remark}

Now we characterize the abelian groups $G$ with finite $\CC(G)$. 

\begin{lemma} \label{almost-divisible}
Let $G$ be an abelian group. Then the following conditions are equivalent: 
\begin{itemize}
\item[(a)]  $G/G^1$ is finite; 
\item[(b)] $G/D(G)$ is finite; 
\item[(c)] the Hausdorff reflection of $(G,\gamma_G)$ is finite; 
\item[(d)] $\CC(G)$ is finite. 
\end{itemize}
\end{lemma}
\begin{proof}
(a)$\Rightarrow$(b) As $G^1$ has finite index in $G$, $G^1$ admits no proper finite-index subgroup. Consequently, $G^1$ is divisible. Since $D(G)\subseteq G^1$, we conclude that $G^1=D(G)$.

\smallskip
(b)$\Rightarrow$(a) Is clear, since $D(G)\subseteq G^1$.

\smallskip
(a)$\Leftrightarrow$(c) Is given by Lemma \ref{res-fin<->G1=0}.

\smallskip
(a)$\Rightarrow$(d) Follows from Remark \ref{G1}.

\smallskip
(d)$\Rightarrow$(a) Since $\CC(G)$ is finite, $G^1$ has finite index in $G$.
\end{proof}

If the equivalent conditions of Lemma \ref{almost-divisible} hold true for an abelian group $G$, then $G=D(G)\times F$, where $F$ is a finite abelian group. This is why we call such a group {\em almost divisible}. 

\medskip
For reader's convenience, we collect in the next face some easy to prove properties of the profinite topology. Note that (a) and (c) are obvious and (b) is given by Lemma \ref{res-fin<->G1=0}.

\begin{fact}\label{profin}
Let $G$ be an abelian group. Then:
\begin{itemize}
\item[(a)] $(G,\gamma_G)$ is totally bounded; 
\item[(b)] $(G,\gamma_G)$ is precompact if and only if $G$ is residually finite;
\item[(c)] every surjective homomorphism $\phi: (G,\gamma_G) \to (H,\gamma_H)$ is continuous and open (i.e., the profinite topology is an ideal functorial topology); in particular, $w(G,\gamma_G)\geq w(H, \gamma_H)$.
\end{itemize}
\end{fact} 

\begin{example}\label{QZ}
\begin{itemize}
\item[(a)] Analogously to item (c) of Fact \ref{profin}, the natural topology is an ideal functorial topology. 
\item[(b)] If $D$ is a divisible abelian group, then $\nu_D$ is indiscrete, and so $\gamma_D$ is indiscrete as well. Indeed, $mD=D$ for every $m\in\N_+$.
\item[(c)] Consider $\Z\subseteq\Q$. By item (a) $\nu_\Q$ and $\gamma_\Q$ coincide with the indiscrete topology of $\Q$. Then $\nu_\Q\restriction_\Z$ and $\gamma_\Q\restriction_\Z$ coincide with the indiscrete topology of $\Z$, while $\nu_\Z$ and $\gamma_\Z$ are not indiscrete. This proves that the natural and the profinite topology are not hereditary.
\end{itemize}
\end{example}

We show now that the family of all $\gamma_G$-closed (respectively, $\gamma_G$-dense) subgroups of an abelian group $G$ coincides with the family of all $\nu_G$-closed (respectively, $\nu_G$-dense) subgroups of $G$:

\begin{lemma}\label{closed_subgroups}
Let $G$ be an abelian group and $H$ a subgroup of $G$. Then: 
\begin{itemize}
\item[(a)] $H$ is $\gamma_G$-closed if and only if $H$ is $\nu_G$-closed; 
\item[(b)] $H$ is $\gamma_G$-dense if and only if $H$ is $\nu_G$-dense. 
\end{itemize}
\end{lemma}
\begin{proof}
Since both $\gamma_G$ and $\nu_G$ are ideal, their quotient topologies on $G/H$ coincide with 
$\gamma_{G/H}$ and $\nu_{G/H}$ respectively. 

\smallskip
(a) The subgroup $H$ is $\gamma_G$-closed (respectively, $\nu_G$-closed) if and only if $\gamma_{G/H}$ (respectively, $\nu_{G/H}$) is Hausdorff. By Lemma \ref{Haus}, $\gamma_{G/H}$ is Hausdorff precisely when $\nu_{G/H}$ is Hausdorff, and so $H$ is $\gamma_G$-closed if and only if $H$ is $\nu_G$-closed.

\smallskip
(b) The subgroup $H$ is $\gamma_G$-dense (respectively, $\nu_G$-dense) if and only if $\gamma_{G/H}$ (respectively, $\nu_{G/H}$) is indiscrete. By Lemma \ref{Haus}, $\gamma_{G/H}$ is indiscrete precisely when $\nu_{G/H}$ is indiscrete, and so $H$ is $\gamma_G$-dense if and only if $H$ is $\nu_G$-dense.
\end{proof}

\subsection{The profinite topology vs the Bohr topology}\label{profinite-bohr}

In the following lemma we give known useful properties of the Bohr topology. Let $G$ be an abelian group. For every $\chi\in G^*$ and $0<\varepsilon\leq1$, let $$U_G(\chi,\varepsilon)=\{x\in G: ||\chi(x)||<\varepsilon\}.$$ Then $\{U_G(\chi,\varepsilon):0<\varepsilon\leq1,\chi\in G^*\}$ is a subbase of the neighborhoods of $0$ in $(G,\mathcal P_G)$.

\begin{lemma}\label{PG}
Let $G$ be an abelian group. Then:
\begin{itemize}
\item[(a)] $\mathcal P_G$ is precompact;
\item[(b)] $\mathcal P$ is hereditary and ideal;
\item[(c)]  $w(G,\mathcal P_G) = 2^{|G|}$;
\item[(d)] every subgroup of $G$ is $\mathcal P_G$-closed;
\item[(e)] $d(G,\mathcal P_G)=|G|$.
\end{itemize}
\end{lemma}
\begin{proof}
(a) As noted in the introduction, $\mathcal P_G$ is totally bounded. Moreover, since the characters $\Hom(G,\T)$ separate the points of $G$, it follows that $\mathcal P_G$ is Hausdorff.

\smallskip
(b) Let $H$ be a subgroup of $G$ and $\pi:G\to G/H$ the canonical projection. Let $0<\varepsilon\leq1$ and consider $U_H(\chi,\varepsilon)$ for a character $\chi$ of $H$. Since $\T$ is divisible, there exists an extension $\widetilde\chi$ of $\chi$ to $G$. Then $U_H(\chi,\varepsilon)=H\cap U_G(\chi,\varepsilon)$, and this proves that $\mathcal P$ is hereditary. 

To prove that $\mathcal P$ is ideal, note that $\pi:(G,\mathcal P_G)\to (G/H,(\mathcal P_G)_q)$ is open by definition of quotient topology. 
Since $(\mathcal P_G)_q$ is precompact, being the quotient topology of the precompact topology $\mathcal P_G$ (see (b)), and since $\mathcal P_{G/H}$ is the finest precompact topology on $G/H$ (as noted in the introduction), we can conclude that $\mathcal P_{G/H}\geq(\mathcal P_G)_q$. In particular, $id_{G/H}:(G/H,(\mathcal P_G)_q)\to (G/H,\mathcal P_{G/H})$ is open, and hence $\pi:(G,\mathcal P_G)\to (G/H,\mathcal P_{G/H})$ is open being composition of two open endomorphisms. This shows that $\mathcal P$ is ideal.

\smallskip
(c) In view of a theorem by Comfort and Ross \cite{CR}, $w(G,\mathcal P_G) = |\Hom(G,\T)|$; now applying a result by Kakutani \cite{Ka} we have $|\Hom(G,\T)|=2^{|G|}$.

\smallskip
(d) For every subgroup $H$ of $G$, since $\mathcal P$ is ideal by (a), on the quotient $G/H$ we have $(\mathcal P_G)_q=\mathcal P_{G/H}$, which is Hausdorff by (b). Hence, $H$ is $\mathcal P_G$-closed.

\smallskip
(e) If $D$ is a $\mathcal P_G$-dense subset of $G$, then $\langle D\rangle$ is a $\mathcal P_G$-dense subgroup of $G$. By item (b) $\langle D\rangle$ is also $\mathcal P_G$-closed and so $\langle D\rangle =G$. In particular, $|D|=|G|$.
\end{proof}

We compare now the profinite topology with the natural topology and the Bohr topology, starting from the relatively easier relation given by Lemma \ref{P,nu>gamma}.

\begin{lemma}\label{P,nu>gamma} 
In the lattice $\mathfrak Ft$ of all functorial topologies, $\gamma\leq\inf\{\nu,\mathcal P\}$.
\end{lemma}
\begin{proof}
Let $G$ be an abelian group. We have to prove that $\gamma_G\leq\inf\{\nu_G,\mathcal P_G\}$.
Since every finite-index subgroup of $G$ contains a subgroup of the form $mG$, one has always $\gamma_G \leq \nu_G$.
Let $H\in\CC(G)$. Since $H$ has finite index in $G$ and $H$ is $\mathcal P_G$-closed by Lemma \ref{PG}(d), we have that $H$ is $\mathcal P_G$-open. Hence, $\gamma_G\leq\mathcal P_G$.
\end{proof}

The following proposition is a fundamental step for the proof of Theorem \ref{inf=gamma}, which is the main result of this section.

\begin{proposition}\label{gamma=P<->bounded}
Let $G$ be an abelian group. The following conditions are equivalent:
\begin{itemize}
\item[(a)] $\gamma_G=\mathcal P_G$;
\item[(b)] $G$ is bounded;
\item[(c)] $G^*$ is bounded;
\item[(d)] $G^*$ is torsion;
\item[(e)] $\nu_G$ is discrete.
\end{itemize}
\end{proposition}
\begin{proof}
The implications (b)$\Rightarrow$(c)$\Rightarrow$(d) and the equivalence (b)$\Leftrightarrow$(e) are obvious.

\smallskip
(a)$\Rightarrow$(d)  By the assumption, for every $\chi\in G^*$, the basic neighborhood $U_G(\chi,1/4)$ contains some $N\in\CC(G)$. Then $N\subseteq \ker\chi$, and so $\ker\chi\in\CC(G)$. Therefore, $mG \subseteq \ker \chi$ for some $m\in \N_+$, i.e., $\chi$ is torsion. Hence, $G^*$ is torsion.

\smallskip
(d)$\Rightarrow$(c) For every $n\in \N_+$ the subgroup $F_n = G^*[n]$ of $G^*$ is closed and $G^* = \bigcup_{n\in\N_+} F_n$ by our hypothesis (d). Since $G^*$ is a compact abelian group, it satisfies the Baire category theorem. Thus, there exists $n\in\N_+$ such that $F_n$ has non-empty interior, hence $F_n$ is open. Since $G^*$ is compact, $F_n$ must have finite index in $G^*$. Therefore, there exists $m \in \N_+$  such that $mG^* \subseteq F_n$, so $mn G^* = 0$, i.e., $G^*$ is bounded.

\smallskip
(c)$\Rightarrow$(b) Assume that $n G^*=0$ for some $n\in\N_+$. To show that $nG=0$, pick an element $x\in G$. Then $\chi(nx)=0$ for every $\chi\in G^*$. It follows that $nx=0$, as it is a well-known fact that the characters of a discrete abelian group separate the points of the group. Hence $nG=0$. 

\smallskip
(d)$\Rightarrow$(a) Let $\chi\in G^*$ and $0<\varepsilon\leq1$. Then $U_G(\chi,\varepsilon)$ contains $\ker\chi$. Since $G^*$ is torsion, there exists $m\in\N_+$ such that $m\chi=0$, that is, $m\chi(G)=0$. Therefore,  $G/\ker\chi\cong \chi(G)$ is finite, so $\ker\chi$ has finite index in $G$. Hence, $\ker\chi$ is open in $(G,\gamma_G)$. Since $\chi\in G^*$ was chosen arbitrarily, this shows that $\mathcal P_G\leq\gamma_G$. Lemma \ref{P,nu>gamma} applies to conclude that $\gamma_G=\mathcal P_G$.
\end{proof}

The following corollary of Proposition \ref{gamma=P<->bounded} shows that the weight of the profinite topology of a bounded abelian group $G$ has the maximal possible value $2^{|G|}$. 

\begin{corollary}\label{gamma-nonmetr}
If $G$ is an infinite bounded abelian group, then $w(G,\gamma_G)=2^{|G|}$. In particular, $(G,\gamma_G)$ is non-metrizable. More precisely, it does not contain infinite compact sets (so in particular, no convergent non-trivial sequences).
\end{corollary}
\begin{proof}
By Lemma \ref{PG}(b), $w(G,\mathcal P_G)=2^{|G|}$.
Moreover, $\gamma_G=\mathcal P_G$ by Proposition \ref{gamma=P<->bounded}. To conclude, recall that the Bohr topology admits no infinite compact sets (see \cite{Glick}).
\end{proof}
 
We have seen in Proposition \ref{gamma=P<->bounded} that for bounded abelian groups the profinite topology coincides with the Bohr topology. In particular, this means that the profinite topology is the infimum of the Bohr topology and the natural topology, since $\mathcal P_G= \gamma_G\leq\nu_G=\delta_G$ for any bounded abelian group $G$.
The next theorem shows that the equality $\gamma_G=\inf\{\nu_G,\mathcal P_G\}$ holds for every abelian group $G$.

\begin{theorem}\label{inf=gamma}
In the lattice $\mathfrak Ft$ of all functorial topologies, $\gamma=\inf\{\nu,\mathcal P\}$.
\end{theorem}
\begin{proof} 
Let $G$ be an abelian group. We have to prove that $\gamma_G=\inf\{\nu_G,\mathcal P_G\}$.

Lemma \ref{P,nu>gamma} gives the inequality $\gamma_G\leq\inf\{\nu_G,\mathcal P_G\}$.

Let now $U$ be a neighborhood of $0$ in $\mathcal P_G$ and $n\in\N_+$. We prove that the typical neighborhood $U+ nG$ of $0$ in $\inf\{\mathcal P_G,\nu_G\}$ is a neighborhood of $0$ in $\gamma_G$ too. To this end consider the canonical projection $\pi:G\to G/nG$. Since $G/nG$ is bounded, $\gamma_{G/nG}=\mathcal P_{G/nG}$ by Proposition \ref{gamma=P<->bounded}. Moreover, $\pi:(G,\mathcal P_G)\to (G/nG,\mathcal P_{G/nG})$ is open, as the Bohr topology is ideal by Lemma \ref{PG}(b). Therefore, $\pi(U)\in\mathcal P_{G/nG}=\gamma_{G/nG}$. Then $H_1\subseteq \pi(U)$ for some $H_1\in\CC(G/nG)$. Consequently $H=\pi^{-1}(H_1)\in\CC(G)$ and in particular $H$ is a neighborhood of $0$ in $\gamma_G$. Since $U+nG=\pi^{-1}(\pi(U))$, it follows that $H\subseteq U+nG$, which proves that $U+nG$ is a neighborhood of $0$ in $(G,\gamma_G)$. This concludes the proof.
\end{proof}

\begin{remark}\label{vD}
\begin{itemize}
\item[(a)] If $G$ and $H$ are infinite bounded abelian groups, then for every continuous map $f: (G,\gamma_G)\to (H,\gamma_H)$ with $f(0)=0$ there exist a {\em homomorphism} $\phi: G \to H$ and an infinite subset $A$ of $G$ containing $0$, such that $f\restriction_A = \phi\restriction_A$. In particular, if $G$ is of exponent $p$ and $H$ is of exponent $q$, where $p$ and $q$ are distinct primes, then there exists no homeomorphisms between $(G,\gamma_G)$ and $(H,\gamma_H)$, considered as topological spaces \cite{DW,kunen}. 
\item[(b)] Item (a) should be compared with the fact that two countable metrizable abelian groups are always homeomorphic considered as {\em topological spaces}; for example, the $p$-adic and the $q$-adic topologies on $\Z$ are homeomorphic. Also the compact spaces $\J_p$ and $\J_q$, provided with their natural topology, are homeomorphic (to the Cantor cube $\{0,1\}^\omega$).
Let us note that in spite of this homeomorphism,  there exists no non-zero homomorphism $\J_p \to \J_q$ (this should be compared to (a)). 
\end{itemize}
\end{remark}

\section{The poset of finite-index subgroups}\label{C(G)}

This section is dedicated to the general problem of studying the poset of finite-index subgroups of an abelian group.
This problem is motivated by the definition of the adjoint algebraic entropy  of an endomorphism $\phi$ of an abelian group $G$, where the family $\CC(G)$ of all finite-index subgroups of $G$ plays a prominent role. Now we recall the precise definition of adjoint algebraic entropy. For $N\in\CC(G)$ and $n\in\N_+$, the \emph{$n$-th $\phi$-cotrajectory} of $N$ is 
$$C_n(\phi,N)=\frac{G}{N\cap\phi^{-1}(N)\cap\ldots\cap\phi^{-n+1}(N)}.$$
The \emph{adjoint algebraic entropy of $\phi$ with respect to $N$} is 
\begin{equation*}\label{H*}
H^\star(\phi,N)={\lim_{n\to \infty}\frac{\log|C_n(\phi,N)|}{n}};
\end{equation*}
This limit exists and it is finite \cite{DGS}. The \emph{adjoint algebraic entropy} of $\phi:G\to G$ is $$\aent(\phi)=\sup\{H^\star(\phi,N): N\in\CC(G)\}.$$

In \cite[Theorem 7.4]{DGS} a dichotomy for the values of the adjoint algebraic entropy is proved; indeed, it can take only the values $0$ and $\infty$.
Moreover, \cite{SZ} is dedicated to the characterization of abelian groups of zero adjoint algebraic entropy and \cite{AGB} to the connection of the adjoint algebraic entropy with the topological entropy.

\subsection{Size and cofinality of $\mathcal C(G)$}

The poset $\CC(G)$ of finite-index subgroups of an abelian group $G$ is trivial if and only if $G$ is divisible. Moreover,
Lemma \ref{almost-divisible} describes the case when $\CC(G)$ is finite; the abelian groups with this property are those that we have called almost divisible.
The next result from \cite{DGS} characterizes narrow abelian groups, that are the abelian groups $G$ with countable $\CC(G)$.

\begin{theorem}\label{TOP:narrow} \label{narrow:char}\emph{\cite[Theorem 3.3]{DGS}}
Let $G$ be an abelian group. Then the following conditions are equivalent: 
\begin{itemize}
\item[(a)] $G$ is narrow;
\item[(b)] $|\CC(G)|<\mathfrak c$;
\item[(c)] $G/pG$ is finite for every prime $p$;
\item[(d)] $G/mG$ is finite for every $m\in\N_+$;
\item[(e)] $\CC(G)$ contains a countable decreasing cofinal chain;
\item[(f)] $\gamma_G=\nu_G$.
\end{itemize}
\end{theorem}

Note that (e) is equivalent to pseudometrizability of $(G,\gamma_G)$. The remarkable dichotomy hidden behind the equivalence between (d) and (f) (i.e., $\CC(G)$ is either countable or has size at least $\cont$), discovered in \cite[Theorem 3.3]{DGS} becomes clear below. It is due to the fact that $|\CC(G)|$ coincides with the cardinality of the torsion
part of a compact abelian group (namely, $|\CC(G)| = |t(\Hom(G,\T))|$). 

\medskip
The next lemma plays a key role in the proofs of the results of this section. 

\begin{lemma}\label{profin*} 
Let $G$ be an abelian group that is not almost divisible. Then:
\begin{itemize}
\item[(a)] $(G,\nu_G)^*= (G,\gamma_G)^*=t(\Hom(G,\T))=t(\Hom(G/G^1,\T))$;
\item[(b)] $|\CC(G)| =w(G,\gamma_G)=\chi(G,\gamma_G)=|t(\Hom(G,\T))|$. 
\end{itemize}
\end{lemma} 
\begin{proof}  
(a)  To prove the first two equalities we intend to show the following chain of inclusions 
\begin{equation}\label{chain}
(G,\nu_G)^*\supseteq(G,\gamma_G)^*\supseteq t(\Hom(G,\T))\supseteq (G,\nu_G)^*.
\end{equation}

Since $\gamma_G\leq \nu_G$, the first inclusion $(G,\gamma_G)^* \subseteq (G,\nu_G)^*$ is obvious. 

To prove the inclusion $ t(\Hom(G,\T))  \subseteq (G,\gamma_G)^*$, note that for every $\chi\in t(\Hom(G,\T))$ there exists $m\in\N_+$ such that $m\chi=0$, so  $\chi(mG) = m \chi(G)=0$. In particular, $\chi(G)$ is finite and so $\ker\chi\in\CC(G)$. Thus, $\chi:(G,\gamma_G)\to\T$ is continuous.

To prove the inclusion $(G,\nu_G)^* \subseteq t(\Hom(G,\T))$, fix a neighborhood $U$ of 0 in $\T$ that contains no non-zero subgroups.  
For every continuous character $\chi:(G,\nu_G)\to\T$ there exists $m\in\N_+$ such that $\chi (mG) \subseteq U$. By the choice of $U$ this yields $\chi(mG) = 0$, i.e., $m\chi=0$. Then $\chi\in t(\Hom(G,\T))$.  
 
From the chain of inclusions \eqref{chain}, we obtain that $(G,\nu_G)^*= (G,\gamma_G)^*=t(\Hom(G,\T))$.

\smallskip
To prove the last equality, we identify first $\Hom(G/G^1,\T)$ with a subgroup of $\Hom(G,\T)$ using the canonical homomorphism  $\pi:G\to G/G^1$. In fact, 
the adjoint homomorphism $\pi^*: \Hom(G/G^1,\T)\to \Hom(G,\T)$ is injective, as $\pi$ is surjective. After this identification, we note that 
the inclusion $t(\Hom(G,\T))\supseteq t(\Hom(G/G^1,\T))$  is clear. So let $\chi\in t(\Hom(G,\T))$. Then there exists $m\in\N_+$ such that $m\chi =0$; in particular, $\ker\chi\supseteq G^1$ and so $\chi$ can be factorized as $\chi=\overline\chi\circ\pi$, 
where $\overline\chi:G/G^1\to\T$ is the character induced by $\chi$. Then $\chi$ can be considered as a torsion character of $G/G^1$, and we have proved that $t(\Hom(G,\T))\subseteq\Hom(G/G^1,\T)$. Then $t(\Hom(G,\T))=t(\Hom(G/G^1,\T))$.

\smallskip
(b) As recalled above, $w(L)=\chi(L)$ for any totally bounded abelian group $L$. So $w(G,\gamma_G)=\chi(G,\gamma_G)$ by Fact \ref{profin}(a). As noted in Lemma \ref{res-fin<->G1=0}, $(G/G^1,\gamma_{G/G^1})$ is the Hausdorff reflection of $(G,\gamma_G)$. Therefore, $w(G/G^1,\gamma_{G/G^1})=w(G,\gamma_G)$ and these groups have the same dual group. Then we may assume without loss of generality that $G^1=0$ and so that $(G,\gamma_G)$ is precompact by Fact \ref{profin}(b). Consequently, $w(G,\gamma_G)=|(G,\gamma_G)^*|$, as $w(L)=|L^*|$ for any precompact abelian group $L$ \cite{CR}. To conclude, $(G,\gamma_G)^*=t(\Hom(G,\T))$ by (a).

It remains to prove that $|\CC(G)|=|t(\Hom(G,\T)|$. Let $\CC_c(G)$ be the subfamily of those $N \in \CC(G)$, such that $G/N$ is (finite) cyclic. 
Note that every $N \in \CC(G) $ is a finite intersection of subgroups from $\CC_c(G)$. Moreover, $\CC(G)$ is infinite by the hypothesis that $G$ is not almost divisible, hence $|\CC(G)| = |\CC_c(G)|$.  So it remains to verify that $|\CC_c(G)|=|t(\Hom(G,\T)|$.

Clearly, every $N \in  \CC_c(G)$ gives rise to a character $\chi_N: G \to \T$ by considering
the finite cyclic group $G/N$ as a subgroup of $\T$. This defines an injective map $\CC_c(G)\to t(\Hom(G,\T))$ as the character $\chi_N$ is obviously torsion.   In the opposite direction we just assign to every $\chi \in t(\Hom(G,\T))$ its kernel $N = \ker \chi\in \CC_c(G)$. 
Let us see next that the map $t(\Hom(G,\T)) \to \CC_c(G) $ defined by 
$\chi \mapsto \ker \chi$ is finitely-many-to-one. This will prove that $|\CC_c(G)|= |t(\Hom(G,\T))|$, as $\CC_c(G)$ is infinite.  
Let $\chi, \eta \in t(\Hom(G,\T))$ with $\ker \chi = \ker \eta =N$. Then there exist injective homomorphisms $\chi_1, \eta_1: G/N\to \T$, such that 
$\chi= \chi_1 \circ \pi$ and $\eta= \eta_1 \circ \pi$, where $\pi: G \to G/N$ is the canonical homomorphism. Let $m = |G/N|$, then both $\chi(G)$ and $\eta(G)$ coincide with the unique cyclic subgroup $C$ of $\T$ of order $m$. Hence, there exists an automorphism $\xi$ of $C$, such that $\eta_1 = \xi \circ \chi_1$. Since the automorphism group of $C$ is finite, 
one has only finitely many distinct pairs $\chi,\eta$ with $\ker \chi = \ker \eta =N$. 
\end{proof}

Passing to the general case, the following theorem gives a precise formula for the cardinality and the cofinality  of the family $\CC(G)$ for an infinite abelian group $G$ when $|\CC(G)|\geq\omega$. 

\smallskip
In the sequel, for an infinite cardinal $\kappa$ we let $\log \kappa = \min\{\lambda: 2^\lambda \geq \kappa\}$ the logarithm of $\kappa$. 
 
\begin{theorem}\label{WEIGHT} 
Let $G$ be an abelian group that is not almost divisible. Then 
\begin{equation}\label{chain=}
|\CC(G) |=\chi(G, \gamma_G)= w(G, \gamma_G)=\omega \cdot \sup \{2^{|G/pG|}: p\ \text{prime}\}. 
\end{equation}
\end{theorem}
\begin{proof}
By Lemma \ref{profin*}(b), $|\CC(G) |= \chi(G, \gamma_G)= w(G, \gamma_G)$,  so it remains to prove only the last equality in \eqref{chain=}.

Assume that $G$ is narrow. According to Theorem \ref{narrow:char} this is equivalent to finiteness of $G/pG$ for every prime $p$. Then $\omega \cdot \sup \{2^{|G/pG|}: p\ \text{prime}\}=\omega$. By Theorem \ref{TOP:narrow} and Lemma \ref{profin*}, $G$ narrow is equivalent to $w(G, \gamma_G)=\omega$ and also to $\chi(G,\gamma_G)=\omega$. This proves the desired equality.

Assume now that $G$ is not narrow, and let $\kappa= w(G, \gamma_G)$. By Theorem \ref{narrow:char}, $\kappa\geq\mathfrak c$. Since every finite-index subgroup of $G$ contains a subgroup of the form $mG$, and since Fact \ref{profin}(c) can be applied for every $m\in\N_+$ for the projection $\pi_m: G\to G/mG$, we deduce that  $\kappa\geq w(G,\gamma_{G/mG})$ for every $m\in\N_+$. Therefore,  $\kappa\geq \sup \{w(G, \gamma_{G/mG}): m\in \N_+\}$. On the other hand, if for each $m\in\N_+$ one fixes a base $\mathcal B_m$ of neighborhoods of $0$ in $(G/mG,\gamma_{G/mG})$ of minimum cardinality, then the family of sets $\bigcup_{m\in\N_+} \{\pi_m^{-1}(B): B\in \mathcal B_m\}$ forms a base of neighborhoods of 0 in $(G, \gamma_G)$ of size $\leq \sup \{w(G/mG,\gamma_{G/mG}): m\in \N_+\}$. This proves that 
\begin{equation}\label{2}
w(G, \gamma_G)=\sup\{w(G/mG,\gamma_{G/mG}): m\in \N_+\}.
\end{equation}
Since $G$ is not narrow, by Theorem \ref{narrow:char} there exists $m\in\N_+$ such that $G/mG$ is infinite. For every $m\in \N_+$ such that $G/mG$ is infinite, $w(G/mG, \gamma_{G/mG})= 2^{|G/mG|}$ by Corollary \ref{gamma-nonmetr} and 
\begin{equation}\label{m->p}
\text{there exists a prime $p$ dividing $m$ such that $|G/mG|=|G/pG|$.}
\end{equation}
Indeed, obviously $|G/mG| \geq |G/pG|$ for every such prime $p$. 
Next we note first that if $G/pG$ is infinite, then $|G/p^nG|=|G/pG|$ for every $n\in \N_+$. Moreover, if $m= p_1^{k_1}\ldots p_s^{k_s}$ with distinct primes $p_i$, then 
$mG = \bigcap _{i=1}^s p_i^{k_i}G$, hence $G/mG \hookrightarrow \bigoplus _{i=1}^sG/ p_i^{k_i}G$. Therefore, $|G/mG| \leq \sup_{i=1}^s |G/p_iG|$. 

Hence, \eqref{2} implies the desired equality. 
\end{proof}

\begin{remark} 
When $G$ is not narrow, there is another proof of this theorem requiring a better knowledge of Pontryagin duality. According to Lemma \ref{profin*}(b), $|\CC(G)| =w(G,\gamma_G)=\chi(G,\gamma_G)=|t(\Hom(G,\T))|$, so it suffices to prove that $|t(\Hom(G,\T))|=\omega \cdot \sup \{2^{|G/pG|}: p\ \text{prime}\}.$ Since $G^* = \Hom(G,\T)$ is compact, 
\begin{equation}\label{**}
|t(G^*)|= r(t(G^*))= \sup \{r_p(G^*): p\ \text{prime}\}=\sup \{2^{r_p(G/pG)}: p\ \text{prime}\}=\sup \{2^{|G/pG|}: p\ \text{prime}\}.
\end{equation}
The next-to-the-last equality uses the fact that 
$r_p(G^*)= 2^{r_p(G/pG)}$, which can be  obtained as follows (see also \cite[Proposition 3.3.15]{DPS}): $r_p(G^*)=r_p(G^*[p])$ and the compact group $G^*[p]$ of exponent $p$
is the annihilator of $pG$, so isomorphic to the dual of the group $G/pG$. Since $G/pG \cong \bigoplus _{r_p(G/pG)}\Z(p)$, we conclude that $G^*[p]\cong \Z(p)^{r_p(G/pG)}$. Hence $r_p(G^*)= 2^{r_p(G/pG)}$
in case these cardinals are infinite, otherwise $r_p(G^*)= {r_p(G/pG)}$. Since at least one of the cardinals $G/pG$ is infinite (by the assumption that  $G$ is narrow
and the fact that if some $G/mG$ is infinite, then also some $G/pG$ is infinite as well, see \eqref{m->p} in the above proof). 
\end{remark}

\begin{corollary}
Let $G$ be a residually finite abelian group. If $G$ is not narrow, then: 
\begin{itemize}
\item[(a)] $w(G, \gamma_G)\geq 2^\mathfrak c$ if and only if there exists a prime $p$ such that $|G/pG|\geq \log 2^\mathfrak c$;
\item[(b)] if $|G/pG|< \log 2^\kappa$ for all primes $p$ and an infinite cardinal $\kappa$, then $w(G, \gamma_G)< 2^\kappa$; 
\item[(c)] $[$under $\mathrm{MA}]$ $w(G, \gamma_G)=\mathfrak c$, if $|G/pG|< \mathfrak c$ for all primes $p$. 
\end{itemize}
\end{corollary}
\begin{proof}
(a) If  $|G/pG|\geq \log 2^\mathfrak c$ for some prime $p$, then obviously $2^{|G/pG|}\geq 2^\mathfrak c$, so $w(G, \gamma_G)\geq 2^\mathfrak c$ by Theorem \ref{WEIGHT}. Assume that $w(G, \gamma_G)\geq 2^\mathfrak c$ for all primes $p$. Then obviously $w(G, \gamma_G)=\sup \{2^{|G/pG|}: p\ \text{prime}\}$ by Theorem \ref{WEIGHT}. Assume that all $|G/pG|<  \log 2^\mathfrak c$. Then 
\begin{equation}\label{<}
2^{|G/pG|}< 2^\mathfrak c \mbox{ for all primes}\ p.
\end{equation}
By our assumption, $w(G, \gamma_G)$ cannot coincide with any $2^{|G/pG|}$, hence $\mathrm{cf}(w(G, \gamma_G))=\omega$. Thus the equality $w(G, \gamma_G)=2^\mathfrak c$ is not possible. Since \eqref{<} rules out the strict inequality $w(G, \gamma_G)> 2^\mathfrak c$, we are left with $w(G, \gamma_G)< 2^\mathfrak c$, a contradiction.

\smallskip
(b) As all $|G/pG|<  \log 2^\kappa$, $2^{|G/pG|}< 2^\kappa \text{ for all primes }p$.  Then $w(G, \gamma_G)\leq 2^\kappa \leq \log 2^\kappa$. If $w(G, \gamma_G)=2^\kappa$, then it cannot coincide with any $2^{|G/pG|}$, so $\mathrm{cf}(w(G,\gamma_G))=\omega$, while $\mathrm{cf}(2^\kappa)>\omega$, a contradiction.

\smallskip
(c) Follows from Theorem \ref{narrow:char}(b) and from the fact that $\mathrm{MA}$ yields $ \log 2^\mathfrak c=\mathfrak c$.
\end{proof}

The equality \eqref{**} suggests to consider also the density character $d(G,\gamma_G)$. Let us recall that $d(G,\mathcal P_G)=|G|$ as every subgroup of $G$ is $\mathcal P_G$-closed. In order to compute  $d(G,\gamma_G)$ we observe first that $d(G,\gamma_G)=d(G,\nu_G)$ by Lemma \ref{closed_subgroups}.  

\begin{theorem}\label{density} 
Let $G$ be an infinite residually finite abelian group. Then 
\begin{equation}\label{***}
\log |G| \leq w(G, \nu_G) = d(G,\gamma_G)=d(G,\nu_G)=\omega \cdot \sup \{{|G/pG|}: p\ \text{prime}\}.
\end{equation}
\end{theorem}
\begin{proof}
The equality $d(G,\gamma_G)=d(G,\nu_G)$ follows from item (b) of Lemma \ref{closed_subgroups}. To prove the last equality in \eqref{***} fix an $m\in \N_+$. Let $D$ be a $\nu_G$-dense subset of $G$ of size $d(G,\nu_G)$. Then its image in $G/mG$ under the canonical homomorphism $\pi: G \to G/mG$ is a dense subset of the discrete group $G/mG$,  so $\pi(D) = G/mG$. Thus $d(G,\nu_G)= |D| \geq |G/mG|$. Since $(G,\nu_G)$ is an infinite Hausdorff group, $d(G,\nu_G) \geq \omega$. This proves the inequality  $d(G,\nu_G)\geq \omega \cdot \sup \{{|G/pG|}: p\ \text{prime}\}. $ 
 
 To prove the opposite inequality, for every $m\in \N_+$ fix a subset $D_m$ of $G$ such that 
\begin{equation}\label{x}
D_m + mG =G\ \  \mbox{ and }\ \  |D_m| = |G/mG|.
\end{equation}
Then the subgroup $H_m$ of $G$ generated by $D_m$ has size $|H_m| =\omega \cdot |G/mG|$. 
So for the subgroup $ H = \sum_{m\in \N_+} H_m$ of $G$ one has $|H| \leq  \omega \cdot \sup \{{|G/mG|}: m\in\N_+\}. $ As noted in \eqref{m->p}, if $G/mG$ is infinite, there exists a prime $p$ dividing $m$ such that $|G/mG|=|G/pG|$. Therefore,
$|H| \leq  \omega \cdot \sup \{{|G/pG|}: p\ \text{prime}\}. $ Since \eqref{x} obviously implies $H + mG = G$ for every $m\in \N_+$, 
$H$ is $\nu_G$-dense in $G$. This proves the last equality in \eqref{***}.

The equality $w(G, \nu_G) = d(G,\nu_G)$ follows from the equalities $w(G,\nu_G) = d(G,\nu_G)\chi(G,\nu_G)$
and $\chi(G,\nu_G)=\omega$. The first inequality in \eqref{***} follows from the inequality $|G|  \leq  2^{w(G,\nu_G)}$.
 \end{proof}

\begin{remark}\label{Rem:w_vs_d} One can relax the condition ``infinite residually finite" to non-almost divisible, as in the previous theorem and lemma. This will be enough to ensure $d(G,\nu_G) \geq \omega$ (then $d(G,\nu_G)$ will have an infinite Hausdorff quotient group, so that $d(G,\nu_G) \geq \omega$ will hold anyway).   
\end{remark}

Theorem \ref{density} gives a precise value of the density character $d(G,\gamma_G)=d(G,\nu_G)$, but in many cases the cardinal
given in \eqref{***} coincides with $|G|$ (as in the  case of the Bohr topology). In the next example we show that these cardinals need not coincide
in general and  the gap may be  as big as possible (i.e.,  $|G|=  2^{d(G,\gamma_G)}$).

\begin{example} Let $p$ be a prime. Let  $B = \bigoplus_{n\in\N_+} \Z(p^n)$,  $\overline{B} = \prod_{n\in\N_+} \Z(p^n)$   and   $ G = t(\overline{B} )$. Then  one can easily see that  
$p^nG + B = G$. This means that  $B$  is $\nu_G$-dense in $G$ (as the natural topology of $G$ coincides with its $p$-adic one). Since $B$ is countable, this proves that $d(G,\nu_G) < |G|$ (as $G$ has size $\cont$).
\end{example}


\subsection{When $\mathcal C(G)$ is cofinal in the poset $\mathcal S(G)$}

Let us recall the next notion following \cite{DP}:

\begin{deff}
 An abelian group $G$ is {\em strongly non-divisible} if no proper quotient of $G$ is divisible.  \end{deff}

Clearly, an abelian group is strongly non-divisible if and only if each of its quotient is reduced. 

\smallskip
In the next theorem we describe a subclass $\mathfrak S$ of the class of all residually finite abelian groups defined by the property that $\mathcal C(G)$ is cofinal in the poset $\mathcal S(G)$ of all subgroups of $G$ ordered by inclusion. This remarkable class can be described also by many other equivalent properties: 

\begin{theorem}\label{str_div} 
Let $G$ be an abelian group. Then the following conditions are equivalent: 
\begin{itemize}
  \item[(a)] every quotient of $G$ is residually finite;
  \item[(b)]  every subgroup of $G$ is $\gamma_G$-closed; 
  \item[(c)]  every subgroup of $G$ is $\nu_G$-closed; 
  \item[(d)]  $G$ is strongly non-divisible;
  \item[(e)] $n=r_0(G)$ is finite and for every subgroup $F\cong \Z^n$ of $G$ one has $G/F \cong \bigoplus _p B_p$, where each $B_p$ is a bounded abelian $p$-group;
  \item[(f)] $n=r_0(G)$ is finite and for every subgroup $F$ of $G$ of rank $n$ one has $G/F \cong \bigoplus _p B_p$, where each $B_p$ is a bounded abelian $p$-group;
  \item[(g)]  $\mathcal C(G)\setminus \{G\}$ is cofinal in the poset $\mathcal S(G)$ of all subgroups of $G$. 
\end{itemize}
\end{theorem}
\begin{proof} 
We prove first the equivalence of (a), (b), (c) and (d). 

\smallskip
(a)$\Rightarrow$(b) Let $H$ be a subgroup of $G$. Then $G/H$ is residually finite by hypothesis. By Lemma \ref{Haus},  $(G/H,\gamma_{G/H})$ is Hausdorff, so $\{0_{G/H}\}$ is $\gamma_{G/H}$-closed in $G/H$. Since $H$ is an inverse image of $\{0_{G/H}\}$ under the continuous canonical homomorphism $G \to G/H$, we conclude that $H$ is $\gamma_{G}$-closed. 

\smallskip
(b)$\Leftrightarrow$(c) Follows from Lemma \ref{closed_subgroups}(a).

\smallskip
(c)$\Rightarrow$(d) Assume that $H$ is a proper subgroup of $G$. To prove that $G/H$ is not divisible, one deduces from (c) that $H$ is not $\nu_G$-dense. So there exists $m \in \N_+$ such that $H + mG \ne G$. Hence $m(G/H) \ne G/H$. 

\smallskip
(d)$\Rightarrow$(a) We prove first that (d) implies $G^1 = 0$.  Let $0 \ne x \in G$. Using Zorn's Lemma one can find a subgroup $H$ of $G$ with $x\not \in H$ and maximal with respect to this property. Then the non-zero element $\bar x  = x + H$  of the quotient $G/H$ is contained in every non-trivial subgroup of $G/H$. Hence  $H$ is quasicyclic,  i.e., isomorphic to a subgroup of Pr\" ufer's group $\Z(p^\infty)$ for some prime $p$. Since $G$  is strongly non-divisible, $G/H$ cannot be divisible, so it must be finite. Hence $H \in \mathcal C(G)$. This proves that $x\not \in \bigcap_{H \in \mathcal C(G)}H={G^1}$ by \eqref{G^1}. Therefore, $G^1 = 0$. 

Now suppose that $H$ is a subgroup of $G$. Since the property (d) is obviously preserved by  quotients, we conclude that $G/H$ is strongly non-divisible as well, so we can apply the first part of this argument to conclude that $G/H$ is residually finite. 

\smallskip
Our next aim is to prove the chain of implications (d)$\Rightarrow$(e)$\Rightarrow$(f)$\Rightarrow$(g)$\Rightarrow$(d) that will establish the equivalence of (d), (e), (f) and (g) and will conclude the proof of the theorem. 

\smallskip
(d)$\Leftrightarrow$(e) Was proved in  \cite{DP}. Here we need only the implication (d)$\Rightarrow$(e). Since the reader may have no easy access to \cite{DP}, we provide a complete proof of this implication.  Assume for a contradiction that $r_0(G)$ is infinite and fix a free subgroup $L$ of $G$ of infinite rank. Then there exists a surjective homomorphism $f:L \to \Q$. Since $\Q$ is divisible, one can extend $f$ to a surjective homomorphism $G\to \Q$, a contradiction since $G$ is strongly non-divisible. 
Fix any  subgroup $F\cong \Z^n$ of $G$. Since $r_0(G) = r_0(F)$, the quotient $G/F$ is torsion. Let  $G/F \cong \bigoplus _p B_p$, where each $F_p$ is an abelian $p$-group.
Assume for a contradiction that $B_p$ is not bounded for some prime $p$. Since $G$ is strongly non-divisible, $B_p$ is reduced. Fix a basic subgroup $B$ of $B_p$. 
Since $B_p$ is unbounded and reduced, this implies that the subgroup $B$ is unbounded too. Since $B$ is a direct sum of cyclic subgroups, we deduce that 
$B$ contains a subgroup $H\cong \bigoplus _n \Z(p^n)$. Fix a surjective homomorphism $h: H \to \Z(p^\infty)$.    Since $ \Z(p^\infty)$ is divisible, one can extend $h$ to a surjective homomorphism $G\to  \Z(p^\infty)$, a contradiction since $G$ is strongly non-divisible. 

\smallskip 
(e)$\Rightarrow$(f) It suffices to note that every subgroup $F$ of $G$ of rank $n$ contains a subgroup isomorphic to $\Z^n$
and that a quotient of a group of the form $\bigoplus _p B_p$ (with each $B_p$ a bounded abelian $p$-group) has the same form. 

\smallskip
(f)$\Rightarrow$(g) We shall check below that the property (f) is inherited by all quotients of $G$. Moreover, there is an (order preserving) bijection between 
the proper finite-index subgroups of $G$ containing a given proper subgroup $H$ of $G$ and the poset $\mathcal C(G/H)\setminus \{G/H\}$. 
Therefore, to prove (f)$\Rightarrow$(g) it suffices to show that (f) implies $\mathcal C(G)\ne \{G\}$, i.e., $G$ has a proper finite-index subgroup. 
Fix any  subgroup $F\cong \Z^n$ of $G$ such that the quotient $G/F$ has the form $G/F= \bigoplus _p B_p$, where each $B_p$ is a bounded abelian $p$-group.
If $G/F\ne 0$, then at least one $B_p \ne 0$, so $B_p$ has a proper finite-index subgroup $N_p$. Then the inverse image of the subgroup
$N_p\oplus \bigoplus_{q\ne p}B_q$ of $G/F$ under the canonical projection $G \to G/F$ is a proper finite-index subgroup of $G$. 
Now assume that $G/F=0$. But this means that $G = F \cong \Z^n$, so $G$ has plenty of proper finite-index subgroups. 

Let us check now that  the property (f) is inherited by all quotients of $G$. Let $G/H$ be a quotient of $G$ with $r_0(G/H)= m \leq n = r_0(G)$. 
Fix any  subgroup $F_1\cong \Z^m$ of $G/H$. We have to see that the quotient $(G/H)/F_1$ has the form $\bigoplus _p B^*_p$, where each $B^*_p$ is a bounded abelian $p$-group. Let $a_1, \ldots, a_m$ be a set of free generators of $F_1$. Pick $x_1, \ldots, x_m\in G$ such that $x_i + H = a_i$  in $G/H$ for $i= 1, \ldots, m$. Then 
$x_1, \ldots, x_m$ form an independent subset of $G$ of size $m$. Since $r_0(G) = n \geq m$, we can complete this independent set to a maximal 
independent subset $X = \{x_1, \ldots, x_n\}$ of $G$ and let $F = \langle x_1, \ldots, x_n\rangle$. Then $F \cong \Z^n$ and  the canonical homomorphism $G \to G/H$ takes $F$ onto $F_1$. Therefore $(G/H)/F_1$ is isomorphic to a quotient of $G/F$. By (f),  $G/F= \bigoplus _p B_p$, where each $B_p$ is a bounded abelian $p$-group. Hence, its quotients have the same form. 

\smallskip
(g)$\Rightarrow$(d) Let $H$ be a proper subgroup of $G$. Then there exists $N\in\CC(G)$ such that $H\subseteq N$. Consequently, $N/H\in\CC(G/H)$ and so $G/H$ is not divisible.
\end{proof}


\section{Functorial topologies vs subcategories of $\Ab$}\label{constr-sec}

We enrich here our supply of examples of functorial topologies. 

\begin{example}\label{examples}
Let $G$ be an abelian group.
\begin{itemize}
 \item[(a)] The \emph{$p$-adic topology} $\nu^p_G$ of $G$ has the family of subgroups $\{p^nG:n\in \N_+\}$ as a base of the neighborhoods of $0$.
 \item[(b)] The \emph{pro-$p$-topology} $\gamma_{G}^{p}$ has the family of all subgroups $H$ of $G$ with $G/H$ a finite abelian $p$-group as a base of the neighborhoods of $0$. Obviously, $\gamma_G =\sup_p \gamma_{G}^{p}$. 
  \item[(c)] The \emph{$p$-Bohr topology} $\mathcal P^p_G$ is the initial topology of all homomorphisms $G \to \Z(p^\infty)$, where $\Z(p^\infty)$ is equipped with the topology inherited from $\T$.
 \item[(d)] A topological group $G$ is called {\em $\aleph_0$-bounded}, if for every non-empty open subset $U$ of $G$ there exists a countable subset $A$ of $G$ with $G = U +A$. The $\aleph_0$-bounded groups were introduced and studied by Guran \cite{Gu}. 

The class $\mathcal G$ of all $\aleph_0$-bounded groups contains the class of totally bounded groups and $\mathcal G$ is stable under taking products, subgroups and quotients. In particular, every abelian group $G$ admits a maximal $\aleph_0$-bounded topology that we shall denote by $\Gs_G$ (for more properties of these topologies see \cite{LT} or \cite{D_SUN}). 
 \item[(e)] The \emph{pro-countable topology} $\varrho_G$ has all subgroups of countable index  as a base of the neighborhoods of $0$. Obviously, $\varrho \leq \mathfrak G$.
\end{itemize}
\end{example}

In the large lattice $\mathfrak F t$ of functorial topologies one has the following diagram.

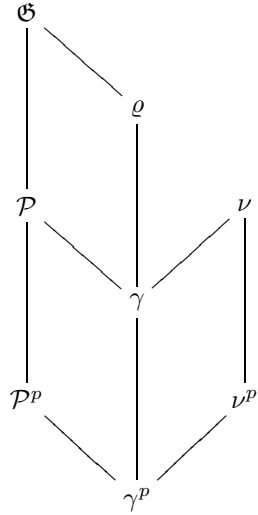
\begin{equation*}
\xymatrix{
\Gs\ar@{-}[dd] \ar@{-}[dr] & &\\
& \varrho \ar@{-}[dd] & \\
\mathcal P \ar@{-}[dr] \ar@{-}[dd]& & \nu \ar@{-}[dd] \ar@{-}[dl]\\
& \gamma \ar@{-}[dd] & \\
\mathcal P^p \ar@{-}[dr]& & \nu^p \ar@{-}[dl] \\
& \gamma^p & \\
}
\end{equation*}

For every prime $p$ and an abelian group $G$ consider the subset 
$$
\mathcal C_p (G) = \{N\in  \mathcal C (G): |G/N| \mbox{ is a power of }p\}
$$
of $\CC(G)$. In analogy to narrow abelian groups one can introduce the following notion.

\begin{deff}
For a prime $p$, an abelian group $G$ is said to be \emph{$p$-narrow} if $\CC_p(G)$ is countable.
\end{deff}

Clearly, the subsets $\mathcal C_p (G)$ generate $\mathcal C(G)$ in the sense that every $N \in \mathcal C (G)$ is a finite intersection of subgroups $N_p \in \mathcal C_p (G)$.

In analogy to the first Ulm subgroup, we introduce
$$
G^1_p=\bigcap_{n\in\N_+}p^n G.
$$

\begin{deff}
Let $p$ be a prime. An abelian group $G$ is \emph{residually $p$-finite} if $G$ is isomorphic to a subgroup of a direct product of finite abelian $p$-groups.
\end{deff}

It is easy to see that $G^1_p=\bigcap_{N\in\CC_p(G)}N$, so $G$ is residually $p$-finite if and only if $G^1_p=0$.

\begin{remark}\label{p-}
The results of the previous sections can be ``localized at $p$'', i.e., for any abelian group $G$,
\begin{itemize}
\item[(a)] $\gamma^p_G=\mathcal P^p_G$ if and only if $G$ is a bounded abelian $p$-group, if and only if $\nu^p_G=\delta_G$;
\item[(b)] $\gamma^p=\inf\{\nu^p,\mathcal P^p\}=\inf\{\nu^p, \gamma\}=\inf\{\mathcal P^p, \gamma\}$;
\item[(c)] $G$ is $p$-narrow if and only if $G/pG$ is finite, if and only if $\gamma^p_G=\nu^p_G$;
\item[(d)] if $\CC_p(G)$ is infinite, then $|\CC_p(G)|=\chi(G,\gamma^p_G)=w(G,\gamma^p_G)=2^{|G/pG|}$;
\item[(e)] if $G$ is infinite and residually $p$-finite, then $\log|G|\leq w(G,\nu^p_G)=d(G,\gamma^p_G)=d(G,\nu^p_G)=|G/pG|$.
\end{itemize}
\end{remark} 


\subsection{The equalizer of two functorial topologies}

Inspired by the results of the previous sections, we consider now functorial topologies in general. In particular, we construct classes of abelian groups starting from functorial topologies. The starting example is that of narrow abelian groups, described by Theorem \ref{narrow:char} as those abelian groups for which the profinite topology coincides with the natural topology (see Example \ref{narrow-eq} below for more details).

\smallskip
Let $\mathcal T$ and $\mathcal S$ be functorial topologies. Define the \emph{equalizer} of $\mathcal T$ and $\mathcal S$ as 
$$\mathcal E(\mathcal T,\mathcal S)=\{G\in\Ab:\mathcal T_G=\mathcal S_G\}.$$
Obviously, $0\in\mathcal E(\mathcal T,\mathcal S)$.

The following are the basic properties of the equalizer. 

\begin{lemma}\label{equalizer}
Let $\mathcal T$ and $\mathcal S$ be functorial topologies. Then:
\begin{itemize}
\item[(a)] $\mathcal E(\mathcal T,\mathcal S)$ is stable under taking finite products; 
\item[(b)] if $\mathcal T$ and $\mathcal S$ are hereditary, then $\mathcal E(\mathcal T,\mathcal S)$ is stable under taking subgroups;
\item[(c)] if $\mathcal T$ and $\mathcal S$ are ideal, then $\mathcal E(\mathcal T,\mathcal S)$ is stable under taking quotients.
\end{itemize}
\end{lemma}
\begin{proof}
(a), (b) and (c) follow directly respectively from Lemma \ref{functorial-basic}(a), (c) and (d).
\end{proof}

So the equalizer of two functorial topologies is always stable under taking finite products, but it may fail to be stable under taking arbitrary products, as item (d) of Example  \ref{narrow-eq} shows.

\medskip
Note that the class $\mathcal E(\mathcal T,\delta)$ is precisely the class $\mathcal C ( \mathcal T)$ introduced and studied in \cite{BM}. Following the terminology used in \cite{BM}, we say that a class of abelian groups is a \emph{discrete class} if it is stable under isomorphic groups, finite products and subgroups; moreover, it is an \emph{ideal discrete class} if it is stable also under quotients.

So Lemma \ref{equalizer} says that $\mathcal E(\mathcal T,\mathcal S)$ is a discrete class if $\mathcal T$ and $\mathcal S$ are hereditary, and $\mathcal E(\mathcal T,\mathcal S)$ is an ideal discrete class if $\mathcal T$ and $	\mathcal S$ are hereditary and ideal.

\medskip
In the sequel we see various examples of classes of abelian groups described as equalizers of pairs of functorial topologies. To this end we need some more notation. Let $\mathcal N$ be the class of all narrow abelian groups and for every prime $p$ let $\mathcal N_p$ be the class of all $p$-narrow abelian groups. Since every $N \in \mathcal C (G)$
is a finite intersection of subgroups $N_p \in \mathcal C_p (G)$, we have that 
$$
\mathcal N = \bigcap _p \mathcal N _p.
$$

\begin{example}\label{narrow-eq}
\begin{itemize}
\item[(a)] By Theorems \ref{narrow:char} and \ref{inf=gamma}, 
$$
\mathcal N =\mathcal E(\gamma,\nu)=\mathcal E(\nu,\mathcal P).
$$ 
Hence, the class $\mathcal N$ is closed under taking quotients; indeed, both the profinite and the natural topology are ideal, and so Lemma \ref{equalizer}(c) applies.
\item[(b)] The class $\mathcal N$ contains all divisible abelian groups $D$, since both $\gamma_D$ and $\nu_D$ coincide with the indiscrete topology of $D$ as noted in Example \ref{QZ}(b).
\item[(c)] In view of item (b), $\mathcal N$ is not closed under taking subgroups, as every abelian group is subgroup of a divisible abelian group. So it suffices to consider $G=\Z(p)^{(\N)}$, where $\gamma_G<\nu_G=\delta_G$, and the divisible hull of $G$. 
\item[(d)] Moreover, $\mathcal N$ is not stable under taking infinite products, as $\Z\in\mathcal N$, while $\Z^\N\not\in\mathcal N$.
\item[(e)] According to Proposition \ref{gamma=P<->bounded}, $\mathcal E(\gamma,\mathcal P)=\{\text{bounded abelian groups}\}$.
\item[(f)]  By Remark \ref{p-}, $\mathcal N_p =\mathcal E(\gamma^p,\nu^p)=\mathcal E(\nu^p,\mathcal P^p),$ and
\item[(g)]  $\mathcal E(\gamma^p,\mathcal P)=\{\text{bounded abelian $p$-groups}\}$. 
\end{itemize}
\end{example}


Item (a) of the next corollary is precisely \cite[Theorem 2.3]{BM}.

\begin{corollary}\label{NewCoro} 
Let $\mathcal T$ be a functorial topology. Then the classes $\mathcal E(\mathcal T,\delta)$ and $\mathcal E(\mathcal T,\iota)$ are stable under isomorphisms and finite products. Moreover: 
\begin{itemize}
\item[(a)] $\mathcal E(\mathcal T,\delta)$ is stable under taking subgroups (i.e., it is a discrete class); if $\mathcal T$ is ideal, then $\mathcal E(\mathcal T,\delta)$ is stable also under taking quotients (i.e., it is an ideal discrete class);
\item[(b)] $\mathcal E(\mathcal T,\iota)$ is stable under taking quotients; if $\mathcal T$ is hereditary, then $\mathcal E(\mathcal T,\iota)$ is stable also under taking subgroups (i.e., it is an ideal discrete class).
\end{itemize}
\end{corollary}
\begin{proof} 
The stability under taking finite products is given by Lemma \ref{equalizer}(a), while the stability under isomorphisms is obvious. 

\smallskip
(a) Let $G\in\mathcal E(\mathcal T,\delta)$ and let $H$ be a subgroup of $G$. Then $\mathcal T_H\geq\mathcal T_G\restriction_H$ by Lemma \ref{functorial-basic}(c). Since $\mathcal T_G\restriction_H=\delta_G\restriction_H=\delta_H$, we can conclude that $\mathcal T_H=\delta_H$, and hence $H\in\mathcal E(\mathcal T,\delta)$. If $\mathcal T$ is ideal, then $\mathcal E(\mathcal T,\delta)$ is stable under taking quotients by Lemma \ref{equalizer}(c).

\smallskip
(b) Let $G\in\mathcal E(\mathcal T,\iota)$ and let $H$ be a subgroup of $G$. Then $\mathcal T_{G/H}\leq (\mathcal T_G)_q$ by Lemma \ref{functorial-basic}(d). Since $(\mathcal T_G)_q=(\iota_G)_q=\iota_H$, we have $\mathcal T_{G/H}=\iota_{G/H}$, that is $G/H\in\mathcal E(\mathcal T,\iota)$. If $\mathcal T$ is hereditary, then $\mathcal E(\mathcal T,\iota)$ is stable under taking subgroups by Lemma \ref{equalizer}(b).
\end{proof}

\begin{example}\label{ex}
\begin{itemize}
\item[(a)] $\mathcal E(\delta,\iota)=\mathcal E(\varrho,\iota)=\mathcal E(\Gs,\iota) $ is the singleton class consisting of the trivial group $\{0\}$; 
\item[(b)] $\mathcal E(\nu,\delta)=\{\text{bounded abelian groups}\}$;
\item[(c)] $\mathcal E(\gamma,\delta)=\mathcal E(\mathcal P,\delta)= \{\text{finite abelian groups}\}$;
\item[(d)] $\mathcal E(\nu,\iota)=\mathcal E(\gamma,\iota)=\{\text{divisible abelian groups}\}$;
\item[(e)] $\mathcal E(\Gs,\delta)=\mathcal E(\varrho,\delta)=\{\text{countable abelian groups}\}$;
\item[(f)]  $\mathcal E(\nu^p,\delta)=\{\text{bounded abelian $p$-groups}\}$;
\item[(g)] $\mathcal E(\gamma^p,\delta)=\mathcal E(\mathcal P^p,\delta)= \{\text{finite abelian $p$-groups}\}$;
\item[(i)] $\mathcal E(\nu^p,\iota)=\mathcal E(\gamma^p,\iota)=\{\text{$p$-divisible abelian groups}\}$.
\item[(j)] $\mathcal E(\mathcal P^p,\iota)=\{\text{torsion abelian groups without non-trivial $p$-torsion elements}\}$.
\end{itemize}
\end{example}

 A careful look at the examples reveals the following connection between topologies and classes of abelian groups. 
Given the class $\mathcal B$ of all bounded abelian groups, one can obtain the natural topology $\nu_G$ of an abelian group $G$ as the group topology on 
$G$ having as base of the neighbourhoods of $0$ the family $\{N\leq G:G/N\in\mathcal B\}$.
Analogously, if $\mathcal F$ is the class of all finite abelian groups, the profinite topology $\gamma_G$ of any abelian group $G$ is the group topology on $G$ which has as a base of the neighbourhoods of $0$ the family $\{N\leq G:G/N\in\mathcal F\}$.

This remark is generalized  in the next theorem. According to Corollary \ref{NewCoro}, the properties required for the class $\mathcal C$ are necessary.  

\begin{theorem}\label{tauC}
Let $\mathcal C$ be a discrete class. 
\begin{itemize}
\item[(a)]The family $\mathcal C_G=\{N\leq G:G/N\in\mathcal C\}$ is a base of the neighborhoods of $0$ of a linear group topology $\mathcal T_G^\mathcal C$ on $G$.
\item[(b)] Moreover, $\mathcal T^\mathcal C$ is a linear hereditary functorial topology with $\mathcal E(\mathcal T^\mathcal C,\delta)=\mathcal C$; and
\item[(c)] $\mathcal C$ is an ideal discrete class if and only if $\mathcal T^\mathcal C$ is ideal.
\end{itemize}
\end{theorem}
\begin{proof}
(a) Let $N_1,N_2\in \mathcal C_G$. Then $G/N_1,G/N_2\in\mathcal C$. Since $\mathcal C$ is stable under taking finite products, $G/N_1\times G/N_2\in\mathcal C$. Consider now the embedding $G/N_1\cap N_2\hookrightarrow G/N_1\times G/N_2$. Since $\mathcal C$ is stable under taking subgroups, $G/N_1\cap N_2\in \mathcal C$, and hence $N_1\cap N_2\in C$. This proves that $\mathcal C_G$ is a local base of a group topology $\mathcal T^\mathcal C_G$ on $G$. Clearly, $\mathcal T^\mathcal C_G$ is linear.

\smallskip
(b) To verify that $\mathcal T^\mathcal C$ is a functorial topology, consider a homomorphism $\phi:(G,\mathcal T^\mathcal C_G)\to (G',\mathcal T^\mathcal C_{G'})$ of abelian groups. Then $\phi$ is continuous; in fact, if $N'\in \mathcal C_{G'}$, then $G'/N'\in\mathcal C$. For $N=\phi^{-1}(N')$ we have $G/N\cong \phi(G)/\phi(N)=\phi(G)/N'\cap\phi(G)\cong \phi(G)+N'/N'$. Since $\mathcal C$ is stable under taking subgroups, $G/N\in\mathcal C$ and so $N\in \mathcal C_G$. The equality  $\mathcal E(\mathcal T^\mathcal C,\delta)=\mathcal C$
follows from the definitions. 

We show now that $\mathcal T^\mathcal C$ is hereditary. To this end let $G$ be an abelian group and $H$ a subgroup of $G$. Let $N\in \mathcal C_G$, that is, $G/N\in\mathcal C$. We have $H/H\cap N\cong H+N/N\leq G/N$. Since $\mathcal C$ is stable under taking subgroups, $H/H\cap N\in \mathcal C$, that is, $H\cap N\in \mathcal C_H$. This proves that $\mathcal T^\mathcal C_G\restriction_H=\mathcal T^\mathcal C_H$, i.e., $\mathcal T^\mathcal C$ is hereditary.

\smallskip
(c) Assume now that $\mathcal C$ is stable under taking quotients. Let $G$ be an abelian group, $H$ a subgroup of $G$ and consider the canonical projection $\pi:G\to G/H$. Let $N\in \mathcal C_G$, that is, $G/N\in\mathcal C$. We have that $(G/H)/\pi(N)=(G/H)/(N+H/H)\cong G/N+H$. Since $\mathcal C$ is stable under taking quotients, $(G/H)/\pi(N)\in\mathcal C$, i.e., $\pi(N)\in \mathcal C_{G/H}$. This proves that $(\mathcal T^\mathcal C_G)_q=\mathcal T^\mathcal C_{G/H}$, hence $\mathcal T^\mathcal C$ is ideal.

Now suppose that  $\mathcal T^\mathcal C$ is ideal. Then  $\mathcal C=\mathcal E(\mathcal T^\mathcal C,\delta)$ is stable under taking quotients by Lemma \ref{equalizer}(c). 
\end{proof}

This procedure is standard in the field of functorial topologies, and it completely describes a large class of functorial topologies 
(for further details see \cite[Theorems 2.5 and 2.6]{BM}).

\medskip
In the next example we exhibit a  linear functorial topology that is not  ideal. 

\begin{example} 
Consider the class $\mathcal C=\{\Z^n: n\in \N\}$. More precisely, let ${\mathcal C}$ be the class of all finite-rank free abelian groups. 
Obviously, ${\mathcal C}$ is stable under isomorphic groups, finite products and subgroups. But ${\mathcal C}$ is not stable under taking quotients. By Theorem \ref{tauC}, $\mathcal T^\mathcal C$ is a linear hereditary functorial topology which is not ideal.
\end{example}

\subsection{Subcategories of $\Ab$ as inverse images via functorial topologies}

We showed above that many subcategories of $\Ab$ arise as equalizers of two functorial topologies. Here we consider other ways to generate subcategories of $\Ab$ using functorial topologies (e.g., via inverse images). To start with, let us note that the classes $\mathcal E(\mathcal T,\delta)\ \text{and}\ \mathcal E(\mathcal T,\iota)$ can be obtained also in a different way. Namely, let $\mathcal A$ be a class of topological abelian groups and let 
$$
\mathcal A(\mathcal T)=\{G\in\Ab:(G,\mathcal T_G)\in \mathcal A\}.
$$
In other words, $\mathcal A(\mathcal T)$ is the inverse image of the subcategory $\mathcal A$ along the functor $\mathcal T: \Ab \to \TopAb$, hence it is natural to expect that nice stability properties of $\mathcal A$ will entail nice stability properties of $\mathcal A(\mathcal T)$:  
 
\begin{lemma} 
Let $\mathcal T$ be a functorial topology and let  $\mathcal A$ be a 
class of topological abelian groups stable under taking products, subgroups and finer group topologies. Then:
\begin{itemize}
\item[(a)] if $G_i\in\mathcal A(\mathcal T)$ for every $i\in I$, then $G=\prod_{i\in I} G_i\in\mathcal A(\mathcal T)$;
\item[(b)] if $G\in\mathcal A(\mathcal T)$, then $H\in\mathcal A(\mathcal T)$ for every subgroup $H$ of $G$.
\end{itemize}
\end{lemma}
\begin{proof}
(a) Since $\mathcal T_G\geq\prod_{i\in I}\mathcal T_{G_i}$ by Lemma \ref{functorial-basic}(b), and since $(G, \prod_{i\in I}\mathcal T_{G_i})\in \mathcal A$ according to the hypothesis, it follows that $(G,\mathcal T_G) \in \mathcal A$.

\smallskip
(b) Since $\mathcal T_H\geq \mathcal T_G\restriction_H$ by Lemma \ref{functorial-basic}(c), we deduce that $(H, \mathcal T_G\restriction_H) \in  \mathcal A$ in view of the hypothesis. 
\end{proof}

Moreover, this construction produces a reflection functor $A_\mathcal T:\Ab\to \mathcal A(\mathcal T)$ associating to an abelian group $G$, its biggest quotient $G/H$ in $\mathcal A(\mathcal T)$ (i.e., such that $(G/H,\mathcal T_{G/H})\in\mathcal A$).

\medskip
We can apply this approach for example for
\begin{itemize}
\item[(a)] $\mathcal A= \mathcal H$ the class of Hausdorff topological groups, so $\mathcal H(\mathcal T)=\{G\in\Ab:\mathcal T_G\ \text{is Hausdorff}\}$;
\item[(b)] $\mathcal A = \Delta$ the class of all discrete groups, so  $\Delta (\mathcal T) = \mathcal E(\mathcal T,\delta)$; 
\item[(c)] $\mathcal A = \mathfrak I$ the class of all indiscrete groups, so $\mathfrak I(\mathcal T) = \mathcal E(\mathcal T,\iota)$. 
\end{itemize}

\begin{example}\label{laaaaaast:ex}
 Here we consider examples concerning only $\nu$.
\begin{itemize}
\item[(a)] For the class $\mathfrak T$ of all totally bounded abelian groups, $\mathfrak T(\nu) = \mathcal N$.
\item[(b)] For the class  $\mathfrak C$ (respectively,  $\mathfrak L$) of all compact (respectively, locally compact) abelian groups, Orsatti \cite{O} characterized the classes $\mathfrak C(\nu) $ and $ \mathfrak L(\nu)$.
\end{itemize}
\end{example}


In the case $\mathcal A=\mathcal H$ we give the following

\begin{example}
For $\mathcal A=\mathcal H$, we have:
\begin{itemize}
\item[(a)]$\mathcal H(\mathcal P)=\mathcal H(\varrho)=\mathcal H(\Gs)=\Ab$;
\item[(b)]$\mathcal H(\gamma)=\mathcal H(\nu)=\{\text{residually finite abelian groups}\}$ by Lemma \ref{Haus};
\item[(c)]$\mathcal H(\gamma^p)=\mathcal H(\nu^p)=\{\text{residually $p$-finite abelian groups}\}$ by Remark \ref{p-}.
\end{itemize}
In this case the reflection functor $H_\mathcal T:\Ab\to \mathcal H(\mathcal T)$ associates to an abelian group $G$, its biggest quotient $G/H$ such that $(G/H,\mathcal T_{G/H})$ is Hausdorff.
For example,
\begin{itemize}
\item[(a$'$)] $H_\mathcal P(G)=G$;
\item[(b$'$)] $H_\gamma=H_\nu=G/G^1$;
\item[(c$'$)] $H_{\gamma^p}=H_{\nu^p}=G/G^1_p$.
\end{itemize}
\end{example}

\section{Open problems and final remarks}

The next question should be compared with item (a) of Remark \ref{vD}. 

\begin{question}\label{Ques1} Let $G = \bigoplus_{\omega_1} \Z(2)$ and $H= \bigoplus_{\omega_1} \Z(3)$. 
Are $(G,\Gs_G)$ and  $(H,\Gs_H)$ homeomorphic as topological spaces? What about $(G,\varrho_G)$ and  $(H,\varrho_H)$?
\end{question}

\begin{problem}\label{vD_problem}
Find a sufficient condition for a functorial topology $\mathcal T$ and a pair of infinite abelian groups $G,H$, so that whenever 
 $f: (G,\mathcal T_G)\to (H,\mathcal T_H)$ is a homeomorphism with $f(0)=0$ there exist a {\em homomorphism} $\phi: G \to H$ and an infinite subset $A$ of $G$ containing $0$, such that $f\restriction_A = \phi\restriction_A$. 
\end{problem}

According to item (a) of Remark \ref{vD}, this is the case of the functorial topologies $\gamma$ and $\mathcal P$ when $G$ and $H$ are bounded abelian groups (see \cite{dLD} for more details). 
Note that if the pair $G,H$ from Question  \ref{Ques1} has the ``linearization property" described above with respect to $\Gs$ or $\varrho$, we obtain a negative answer to Question  \ref{Ques1}. 
  
\begin{question}
Is it true that for every functorial topology $\mathcal T$ every infinite abelian group $G$ such that $(G,\mathcal T_G)$ is Hausdorff is necessarily zero-dimensional? 
\end{question}
  
The answer is obviously ``Yes" for all linear topologies. Moreover, it is ``Yes" for $\mathcal T= \mathcal P$ and this follows from results due to van Douwen \cite{vD} and Shakhmatov \cite{Sh}.

\medskip
The next problem concerns compactness-like properties of the functorial topologies. The weakest possible one, namely total boundedness is present in the case of the Bohr topology $\mathcal P$. Nevertheless, it is known that $(G,\mathcal P_G)$ is pseudocompact  for no infinite abelian group $G$ \cite{CS}. 
 Motivated by this fact and item (b) of Example \ref{laaaaaast:ex}, we propose the following general question in the line of Example \ref{laaaaaast:ex}(b): 

\begin{problem}
Study the class $\mathcal A(\mathcal T)$ for a functorial topology $\mathcal T$ when $\mathcal A$ is a class of compact-like abelian groups. 
\end{problem}

\begin{question}
Do there exist a functorial topology $\mathcal T$ and an infinite abelian group $G$ such that $(G,\mathcal T_G)$ is Hausdorff and $|G|>  2^{d(G,\mathcal T_G)}$? 
\end{question}

Let us recall that there exist Hausdorff topological groups with $|G|=  2^{2^{d(G)}}$. 

\medskip
Motivated by item (e) of Example \ref{ex}, one can ask whether the pair $\Gs$, $\varrho$ has an analogous behavior as the pair $\mathcal P$, $\gamma$. So we leave open the following

\begin{problem}
Describe the precise relation between $\Gs$ and $\varrho$, and compute $\mathcal E(\Gs,\varrho)$. 
\end{problem}
 
According to Lemma \ref{PG}, one has 
\begin{equation}\label{Bohr_ci}
|d(G,\mathcal P_G)| = |G|\mbox{ and }|w(G,\mathcal P_G)| = 2^{|G|}
\end{equation}
for every abelian group $G$. These relations depend only on the cardinality of the abelian group $G$. Theorems \ref{WEIGHT} and \ref{density} show that things change for the profinite topology, where the algebraic structure of  the group starts to have some impact 
through the following condition, relevant for the computation of $d(G,\nu_G)$: 
\begin{equation}\label{no-Top}
\mbox{ the set of cardinals } \{|G/pG|: p\in P\}\mbox{ has no top element}.
\end{equation}
Obviously, (\ref{no-Top}) implies that  $d(G,\nu_G)$ is a limit cardinal with $\mathrm{cf}(d(G,\nu_G))=\omega$.  
When (\ref{no-Top}) fails, then $d(G,\nu_G)= |G/pG|$ for some prime $p$, and consequently $w(G,\nu_G)=2^{d(G,\nu_G)}$ by Theorems \ref{WEIGHT} and \ref{density}. 
Since this may occur independently on the cofinality of $|G/pG|$ (that may also be countable), this shows that $\mathrm{cf}(d(G,\nu_G))=\omega$ does not imply (\ref{no-Top}).
Therefore,  $w(G,\gamma_G)= 2^{d(G,\gamma_G)}$ may occur quite often (for example, when (\ref{no-Top}) fails, but not only in that case). 
To give a more precise account on this, let us recall that for a limit cardinal $\lambda$ one puts $2^{<\lambda} =\sup\{2^\mu: \mu < \lambda\}$. 
In these terms, $w(G,\gamma_G)=2^{<d(G,\nu_G)}$ when (\ref{no-Top}) holds (so $d(G,\nu_G)$ is a limit cardinal with $\mathrm{cf}(d(G,\nu_G))= \omega$).
Since the cardinal function $2^{<\lambda}$ is strongly dependent on the cardinal arithmetics, this leaves open the following general problem: 

\begin{problem}
Describe the precise relation between the cardinals $|G|$, $w(G,\gamma_G)$ and ${d(G,\gamma_G)}$ for an abelian group $G$.
\end{problem}

To give a more precise form of this problem, one can present it as a ``realization problem": 

\begin{problem}
Characterize all triples of infinite cardinals $(\lambda, \kappa, \mu)$ such that $\lambda=d(G,\gamma_G)$,  $\kappa=w(G,\gamma_G)$ and $\mu = |G|$ for some abelian group $G$.
\end{problem}

 Obviously, one has to impose $\lambda \leq \mu \leq 2^\lambda$ and $\lambda \leq \kappa \leq 2^\lambda$ on the triples $(\lambda, \kappa, \mu)$. 

As a first step, one may consider the problem of realization of single cardinals. While every infinite cardinal $\lambda$ can be of the form $d(G,\gamma_G)$ for some abelian group $G$ of size $\lambda$ (take the free abelian group $G$ of size $\lambda$, then $d(G,\gamma_G)= \lambda$ and $w(G,\gamma_G)=2^{\lambda}$), a successor non-exponential  
cardinal cannot be realized as $w(G,\gamma_G)$, according to Theorem \ref{WEIGHT}. 

Next comes the realization of a pair of infinite cardinals $(\lambda, \kappa)$ as a pair $(d(G,\gamma_G), w(G,\gamma_G))$ for some abelian group $G$. 
As already seen above, all pairs $(\lambda, 2^\lambda)$ with an arbitrary infinite cardinal $\lambda$ are realizable.  The ``antipodal" condition $\kappa=\lambda$ is discussed in the next remark (the equality $w(G,\gamma_G)=d(G,\gamma_G)$ should be compared with Question \ref{Bohr_ci}).  

\begin{remark}
 Recall that a cardinal $\kappa$ is a \emph{strong limit cardinal} if $2^\lambda<\kappa$ for every cardinal $\lambda<\kappa$; obviously, under the assumption of the Generalized Continuum Hypothesis (GCH), all limit cardinals are strong limit cardinals. If $\kappa$ is a limit cardinal, then $\kappa$ is a strong limit cardinal if and only if $\kappa= 2^{<\lambda}$. 
\begin{itemize}
\item[(a)] One can show that $w(G,\gamma_G)$ is a strong limit cardinal  for a residually finite abelian group $G$ if and only if $w(G,\gamma_G)=d(G,\gamma_G)$. In such a case its cofinality is countable. Similarly, if (\ref{no-Top}) holds true, then $d(G,\gamma_G)$  is an uncountable strong limit cardinal if and only if $w(G,\gamma_G)=d(G,\gamma_G)$.
In particular, no strong limit cardinal $\kappa$ of uncountable cofinality can be realized as $w(G,\gamma_G)$.
\item[(b)] One can realize the pair $(\kappa, \kappa)$ for every  strong limit cardinal $\kappa$ of countable cofinality.
\end{itemize}
\end{remark}

\medskip
We end this section and the paper noting, as suggested by the referee, that the notion of functorial topology need not be confined to discrete abelian groups. 
%
Indeed, a functorial topology on $\TopAb$ can be defined as a functor $\mathcal T:\TopAb\to\TopAb$ such that $U\mathcal T=U$.

In other words, a functorial topology on $\TopAb$ is a functor $\mathcal T:\TopAb\to\TopAb$ such that $\mathcal T(G,\tau)=(G,\mathcal T_{(G,\tau)})$ for any $(G,\tau)\in\TopAb$, where $\mathcal T_{(G,\tau)}$ denotes the topology on $G$, and $\mathcal T(\phi)=\phi$ for any continuous homomorphism $\phi$ in $\TopAb$.

\begin{example}
\begin{itemize}
\item[(a)] If $(G,\tau)$ is a topological abelian group, the \emph{Bohr modification} of $\tau$ is the topology $\tau^+=\sup\{\tau':\tau'\leq\tau,\ \tau'\ \text{totally bounded}\}$; this topology is the finest totally bounded group topology on $G$ coarser than $\tau$. Actually, $\tau^+=\inf\{\tau,\mathcal P_G\}$.  
So the functor $\mathcal P:\TopAb\to \TopAb$, defined on the objects of $\TopAb$ by $\mathcal P(G,\tau)=(G,\tau^+)$, is a functorial topology. Since $\delta_G^+=\mathcal P_G$, this functorial topology extends the functor of the Bohr topology from $\Ab$ to $\TopAb$.
\item[(b)] For a topological group $(G,\tau)$, let $\tau_\gamma$ be the group topology on $G$ having as a local base the $\tau$-open finite-index subgroups of $G$. So the functor $\gamma:\TopAb\to \TopAb$, defined on the objects of $\TopAb$ by $\gamma(G,\tau)=(G,\tau_\gamma)$, is a functorial topology. Since $(\delta_G)_\gamma=\gamma_G$, this functorial topology extends the functor of the profinite topology from $\Ab$ to $\TopAb$.
\end{itemize}
\end{example} 

Following the verifications in Section \ref{profinite-bohr}, it is possible to prove that the inequality $\gamma\leq\mathcal P$ still holds in this more general setting.

On the other hand, for a topological group $(G,\tau)$, let $\tau_\nu$ be the group topology on $G$ having as a local base the $\tau$-open subgroups of $G$ of the form $mG$ for some $m\in\N_+$. Note that $(\delta_G)_\nu=\nu_G$. The next example shows that there exists a topological abelian group $(G,\tau)$ such that the inequality $\tau_\gamma\leq\tau_\nu$ fails.

\begin{example}
Let $p$ be a prime, $G=\Z(p)^\omega$ and let $\tau$ be the product topology on $G$. Then $\tau_\gamma=\tau$, while $\tau_\nu=\iota_G$.
\end{example}

Moreover, the following example shows that the map $\nu:\TopAb\to \TopAb$, defined on the objects of $\TopAb$ by $\nu(G,\tau)=(G,\tau_\nu)$ is not a functorial topology, as it does not send continuous homomorphisms to continuous homomorphisms. 

\begin{example}
Let $p$ be a prime, $G=\bigoplus_\N \Z(p)$ and $H=\bigoplus_{2\N}\Z(p)\oplus\bigoplus_{2\N+1}\Z(p^2)$, so that $G\subseteq H$. Moreover, let $N=0\oplus\bigoplus_{2\N+1}\Z(p)\subseteq G$ and note that $N=pH$. Let $\tau$ be a non-discrete group topology on $G$ such that $N$ is $\tau$-open in $G$. Let $\sigma$ be the group topology on $H$ defined imposing that the embedding $i:(G,\tau)\to (H,\sigma)$ is continuous and open, so that $N$ is $\sigma$-open in $H$. We verify that $i:(G,\tau_\nu)\to (H,\sigma_\nu)$ is not continuous. In fact, $\tau_\nu$ is indiscrete as $\tau$ is not discrete and $m G$ is either $0$ or $G$ for every $m\in\N_+$, while $pH=N$ is a non-trivial proper subgroup of $G$ and it is $\sigma_\nu$-open.
\end{example}

Finally, analogous considerations can be done about the $p$-Bohr topology, the pro-$p$-topology and the $p$-adic topology.

\subsection*{Acknowledgements}

It is a pleasure to thank the referee for her/his unusually careful and competent reading, for the numerous useful comments and for the constructive criticism.

\end{document}